\documentclass[reqno]{amsart}

\usepackage{
  bm,
  bbm,
  color,
  a4wide,
  amssymb,
  amsmath,
  latexsym,
  calligra,
  mathrsfs,
  stmaryrd,
  hyperref,
  graphicx,
  mathtools,
  enumitem,
  tikz,
  booktabs,
  cancel
}
\usepackage[normalem]{ulem}

\usepackage[2cell,arrow,curve,matrix]{xy}

\newcommand{\p}{\mathfrak{p}}
\newcommand{\q}{\mathfrak{q}}

\newcommand{\fa}{\mathfrak{a}}

\newcommand{\fm}{\mathfrak{m}}
\newcommand{\fr}{\mathfrak{r}}

\newcommand{\fC}{\mathfrak{C}}

\newcommand{\SL}   {\mathbf{SL}}

\newcommand{\Mon}  {\mathbf{Mon}}

\newcommand{\CMon} {\mathbf{CMon}}

\newcommand{\cA}{\mathcal A}
\newcommand{\cB}{\mathcal B}
\newcommand{\cC}{\mathcal C}
\newcommand{\cG}{\mathcal G}
\newcommand{\cE}{\mathcal E}
\newcommand{\cF}{\mathcal F}
\newcommand{\cH}{\mathcal H}

\newcommand{\cL}{\mathcal L}

\newcommand{\I}{\mathbb I}
\newcommand{\N}{\mathbb N}

\newcommand{\bH}{\mathbb H}

\newcommand{\sC}{\mathsf{C}}
\newcommand{\sD}{\mathsf{D}}
\newcommand{\sZ}{\mathsf{Z}}

\DeclareMathOperator*{\Pcar} {\mathsf{Pcar}}

\DeclareMathOperator*{\Car}  {\mathsf{Car}}

\DeclareMathOperator*{\Spec} {\mathsf{Spec}}

\DeclareMathOperator {\Hom}  {\mathsf{Hom}}

\DeclareMathOperator {\id}   {\mathsf{Id}}

\DeclareMathOperator {\Aut}  {\mathsf{Aut}}
\DeclareMathOperator {\End}  {\mathsf{End}} 

\DeclareMathOperator*{\colim}{\mathsf{colim}}

\DeclareMathOperator*{\Coex} {\mathsf{Coext}}

\DeclareMathOperator*{\Pcoex}{\mathsf{PCoext}}
\DeclareMathOperator*{\Clev} {\mathsf{Clev}}

\DeclareMathOperator {\iso}  {\xto{\sim}}

\renewcommand{\ker}{\mathsf{Ker}}

\newcommand{\xto}  [1]   {\xrightarrow{#1}}

\newcommand{\Msl}  [1][M]{#1^\mathsf{sl}}
\renewcommand{\sl}       {(-)^\mathsf{sl}}

\newcommand{\qtext}[1]{\quad\text{#1}\quad}
\newcommand{\qqtext}[1]{\qquad\text{#1}\qquad}

\newcommand{\pt}{\, \bullet\, }

\newcommand{\red}[1]{{\color{red}#1}} 
 \let\scong\cong
\renewcommand{\cong}{\;\scong\;}

\newcommand{\capcup}{%
  \mathbin{
  \ooalign{
  \(\displaystyle\cap\)\cr
  \hidewidth
  \hspace{0.7ex}%
  \raisebox{0.1ex}{\scalebox{0.5}{\(\cup\)}}
  \hidewidth\cr
}
}
}

\newtheorem{Core}{Definition}[section]
\newtheorem{Aux}{Remark}[section]
\newtheorem{Convention}{Convention}

\newtheorem{De}[Core]{Definition}
\newtheorem{Th}[Core]{Theorem}
\newtheorem{Pro}[Core]{Proposition}
\newtheorem{Le}[Core]{Lemma}
\newtheorem{Cor}[Core]{Corollary}

\newtheorem{Rem}[Aux]{Remark}

\newtheorem{Const}[Aux]{Construction}

\newtheorem{Conv}[Convention]{Convention}

\makeatletter
\def\part{%
  \newpage
  \vspace*{2em}%
  \@startsection{part}{0}%
  {\z@}%
  {\z@}%
  {.5\linespacing}%
  {\normalfont\LARGE\bfseries\centering}}
  \makeatother

  \author{Ilia Pirashvili}
  \address{University of Galway, Áras De Brún, Gaillimh/Galway, H91 H3CY, Ireland}
  \email{ilia\_p@ymail.com (personal)}
  \email{ilia.pirashvili@universityofgalway.ie (work)}

  \newcommand{\cHCM}{\cH_{Mn}}

\begin{document}
  \title[Deformation Theory of Monoid Schemes II]{Deformation Theory of Monoid Schemes
  II:\\
  Precartesian Coextensions of Commutative Monoids}

  \maketitle

  \begin{abstract}
    This paper is a continuation of \cite{p15}, where we studied precartesian
    coextensions of commutative monoids by systems of abelian groups. In this paper, we
    generalise it to study coextensions by systems of commutative monoids. Among other
    things, we showcase that this version is able to simultaneously generalise Leech's
    version of coextensions and Redei's version, also called Schreier coextensions.

    We show that what going from systems of abelian groups to systems of commutative
    monoids costs us is quasi-inverses in \(\Pcoex(M, \cL)\). Specifically, instead of
    a symmetric categorical group, they now only become symmetric monoidal groupoids.
    Moreover, though not explicitly stated in the body of the paper, another core
    difference is that for a monoid scheme \(X\), regarded as a monoid functor, a
    precartesian coextension no longer need to be a monoid scheme. Thus, \(\Pcoex(X,
    \cL)\) is, as stated, ineffective at studying monoid scheme coextensions. The rest
    of the theory goes through, but requires developing a new cohomological approach.
  \end{abstract}

  \section*{Introduction}

    The theory of group extensions arose in the classical work of Otto Schreir ``Über
    die Erweiterung von Gruppen, I,II''\cite{Schreier} and then developed by Baer
    \cite{Baer} and A. M. Turing \cite{Turing}. These investigations were the main
    motivation for introducing group cohomology by S. Eilenberg and S. MacLane in the
    series of papers entitled ''Cohomology theory in abstract groups I,II,III''
    \cite{Eilenberg_MacLane}. Nowadays, it is a classical part of homological algebra.

    Generalisation of group extensions and group cohomology to monoids has a long and
    interesting history. There are various approaches for this, but we only list two
    such directions, which are relevant for this paper.\newline

    The first one was proposed by Redei \cite{Reidei}. This approach considers a short
    exact sequence of monoids
    \[ 1\to A \to B \xto{\pi} C\to 0 \]
    where the map \(p\) has a section \(s\) (set map only, in general), which satisfy
    some rather strong conditions. For instance, it guaranties that that any element of
    \(B\) can be written as the product \(a s(c)\) in a unique way, where \(a\in A\)
    and \(c\in C\). In particular, this fact implies we can use \(A \times C\) to write
    \(B\) effectively in ``coordinate form''; hence the reference to Descartes.

    The extensions introduced by Redei are known as Schreier extensions of monoids and
    were studied extensively by many authors including H. Inassaridze, R. Strecker
    \cite{Strecker}, R.O. Fulp and J. W. Stepp \cite{Fulp_Stepp}, I. Fleischher
    \cite{Fleischer}, A. Patchkoria \cite{patchkoria, Patchkoria2}, and many others
    \cite{MF-M-P-S}. It was noted very recently in \cite{Manuell} and \cite{p13b} that
    Schreier extensions of monoids are a particular case of fibred categories,
    introduced by A. Grothendieck in 1961. \newline

    The other direction we wish to highlight is Leech's theory of monoid extensions,
    introduced by J. Leech \cite{Leech}. We also refer to a recent and very nice book
    by A, Cegarra and J. Leech \cite{Cegarra_Leech}. Here, we point out a subtle
    difference between these theories: In Redei's theory, we considered morphisms \(B
    \to C\) which satisfied certain properties. In Leech's theory, they instead carry
    additional data, which consists of a collection of groups \((A_c)_{c\in C}\), with
    actions of \(A_c\) on the fibres \(\pi^{-1}(c)\), satisfying some rather strong
    conditions (in particular, regularity).

    The aim of this work is to develop a theory which generalises both approaches
    simultaneously.\\[2em]

    As the name indicates, this is a continuation of a previous paper I published,
    where I started to study the systematic study of monoid schemes \cite{p15}. In that
    paper, I developed a deformation theory of monoid schemes, where I considered
    coextensions by systems of abelian groups. In that, I showed that the coextensions
    \(\Coex(\cF, \cA)\) of a monoid functor \(\cF\) with a system of abelian groups
    \(\cA\) is a symmetric categorical group and equivalent to the one obtained by the
    abelian group homomorphism \([C^0 \to \ker\partial^1]\). I also showed in
    \cite{p15} that in the special case when \(\cF\) was not just a monoid functor but
    a monoid scheme, the coextensions also happened to all be monoid schemes, thereby
    allowing us to get the theory of monoid coextensions ``for free''. In that
    particular case, the theory actually got richer and we showed that coextensions
    actually formed a stack of
    symmetric categorical groups.\\[2mm]

    In this paper, we proceed to study the coextensions of monoid functors and monoid
    schemes with systems of commutative monoids. This is a rather sizeable
    generalisation that requires new tools and techniques. It also establishes some
    links with a previous paper I wrote \cite{p13b}. In that, the link between the
    Grothendieck construction for fibred categories for one-object categories and
    regular Schreier coextensions was made explicit (as mentioned above, this was also
    independently shown recently in \cite{Manuell}), and generalised to prefibrations
    and Schreier coextensions. In particular, paper \cite{p13b} showed that
    prefibrations are enough for most things, and our current paper reflects that
    sentiment, as we will restrict ourselves to that case.

    One difference to \cite{p13b} is that we only consider commutative monoids here.
    But the main observation, the prefibrations tend to be strong enough to build up
    most of
    the theory, holds in this paper as well.\\[2mm]

    The paper is divided in two parts: The affine and the monoid functor case.

    Having defined an abelian coextension of a monoid \(M\) by a system of abelian
    groups \(\cA\) as a surjective homomorphism \(N\to M\) with regular and compatible
    actions of groups \(\cA(m)\) on the fibres \(\pi^{-1}(m)\) in \cite{p15}, see
    Definition~\ref{def:coext_original}, we define a precartesian coextension
    analogously, but instead of regularity, we demand the existence of a precartesian
    element in the fibre, see \ref{def:pcoext}. More specifically, \(N\) is said to be
    a precartesian coextension of \(M\) by a system of commutative monoids \(\cL\), if
    for every \(m\in M\) we can choose and element \(n_M\in \pi^{-1}(M)\) in it's fibre
    so that every other point of that fibre can be uniquely reached by it through the
    action of \(\cL(m)\) on \(\pi^{-1}(m)\). This element is called a basepoint in the
    paper. Note that by no means are such elements unique in general, but are related
    through invertible elements. Of course, we also need a compatibility condition, but
    this is identical to the abelian case.

    The core difference is that we no longer have groups but monoids acting and can
    thus no longer take inverses. The solution to this is often ``simple bookkeeping'',
    as grouping positive and negative parts on either side of an equality is often
    times enough. But it does require a new type of cohomology theory, and some of our
    results are weakened.

    We then proceed to do a little systematic study of such objects, introducing
    concepts such as morphisms, what it means to be split, pushforwards and other basic
    necessities to develop our theory. An important lemma is~\ref{le:split}, which
    states how to classify split precartesian coextensions. As one might assume from
    the name, these are coextensions which are isomorphic to the trivial coextension,
    being just the semi-direct product \(\cL\rtimes M\). Of course, we give an explicit
    description of what is meant by \(-\rtimes -\) in our setting.\\

    We move forward to define and use the pushforward and pullback constructions, which
    allow us to shift coextensions between varying \(\cL\) or \(M\). While, of course,
    important separately as well, it is together that they allow us to define a central
    part of our paper, namely the Baer sum in our context: When fixing both the monoid
    \(M\) and the system of commutative monoids \(\cL\), we can consider the category
    \(\Pcoex(M,\cL)\) of precartesian coextensions. Our first important theorem
    regarding that is the Short~Five~Lemma~\ref{lem:short5_precart}, which makes
    \(\Pcoex(M,\cL)\) into a groupoid. This is arguably a little surprising, as we can
    still get inverses without requiring that our system \(\cL\) be made of groups.

    The Baer sum then endows \(\Pcoex(M,\cL)\) with the structure of a symmetric
    monoidal groupoid.\\

    As is perhaps more expected, we are able to develop a cohomological machinery that
    allows us to study \(\Pcoex(M,\cL)\). We do so by defining the chain groups
    \(\sC^0, \sC^1, \sC^2\) and the first cocycle monoid \(\sZ^1\subseteq \sC^1\). We
    follow it up by one of our main theorems, being
    Theorem~\ref{thm:classification_pcoex}, which states that we can use our
    cohomological machinery to define a category equivalent to the category of
    precartesian coextensions. Or to put it another way: we can classify
    \(\Pcoex(M,\cL)\) using the cohomological tools \(\sZ^1\) and \(\sC^i\). Moreover,
    this respects the Baer-sum, which is just the pointwise multiplication on the level
    of \(\sZ^1\).

    Much as in paper \cite{p15}, we also define versale coextensions, which are
    coextensions of \(M\) from which we have maps into every other coextension of
    \(M\). Note, however, that the maps are not unique, and neither are the versal
    coextensions (hence the term ``versale'', not ``universal'').\\

    This finishes our results in the affine case of precartesian coextensions and we
    move on to the case of precartesian coextensions of monoid functors.

    A precartesian coextension of a monoid functor is just a collection of compatible
    precartesian coextensions at every point. Other than the expected tasks of checking
    that certain constructions in the affine case respect the compatibility conditions,
    the main challenge in this setting is defining the global cocycle monoid, which is
    done in Section~\ref{sec:global_grillet_pcoex}. After proving a bunch of technical
    lemmas (~\ref{lem:phiG_transitive},~\ref{le:Z_norm_pp},~\ref{lem:phiG_monoid_hom}
    and~\ref{le:Z_norm_1}), which together amount to checking a sheaf condition, we
    prove the classifying theorem in the monoid functor case as well in
    Section~\ref{sec:classifying_thm_functor}, specifically
    Theorem~\ref{thm:classification_pcoex_functor}.\\[2mm]

    The paper is organised as follows: We start with a very quick recollection of some
    important facts regarding monoid schemes, posets, poset topologies, and that
    sheaves over poset topologies correspond to functors over the underlying posets in
    Section~\ref{pro:unique_scheme}. We also recall some constructions important for
    this paper, such as cartesian and precartesian elements and Schreier
    coextensions.\\

    After this, the actual paper splits in two parts, being the affine and the monoid
    functor cases.\\[1em]

    For the affine one, we start with the definition and relate it to the previous
    works in this area. We show that our new approach generalises both the coextensions
    by abelian systems and the commutative Schreier coextensions in
    Lemma~\ref{prop:generalises_group_coext} and
    Proposition~\ref{prop:generalises_schreier}, respectively.\\

    We move to Section~\ref{sec:first-results} where we obtain our first important
    results, most prominently Lemmas~\ref{le:split} and~\ref{lem:short5_precart}.

    In particular, we introduce and discuss split coextensions: While ``trivial'',
    these are one of the most important coextension and subsequently, it has it's own
    subsection~\ref{sec:split-coextension}. We show that split coextensions are
    effectively the same thing as semi-direct products and intimately linked to
    cartesian derivations. This is summed up in Lemma~\ref{le:split}.

    We then observe how basepoints move under morphisms which allows us to prove the
    Short~Five~Lemma~\ref{lem:short5_precart}. As one might expect, this lemma is very
    important and, in particular, it will effectively show that precartesian
    coextensions form a groupoid, see Corollary~\ref{cor:Pcoex_groupoid}.\\

    Section~\ref{sec:pushforward_pullback} deals with the pushforward and the pullback
    constructions. These deal with the following situations: Assume we have a
    coextension \(0 \to \cL \to N \xto{\pi} M \to 0\) and maps \(\alpha: \cL \to
    \cL'\), \(M' \to M\). The pushforward allows us to get a coextension \(0 \to \cL'
    \to K \xto{\sigma} M \to 0\) and the pullback a coextension \(0 \to \varphi^*\cL
    \to \varphi^* N \xto{\sigma} M' \to 0\). The details and what exactly \(K\) and
    \(\sigma\) represent are explained in the corresponding
    Subsections~\ref{sec:pushforward} and~\ref{sec:pullback}, but ultimately, they
    allow us to define the so called Baer sum in \(\Pcoex(M, \cL)\) in
    Section~\ref{sec:pcoex}, more specifically, Subsection~\ref{sec:baer_affine}.

    This endows \(\Pcoex(M, \cL)\) with the structure of a symmetric monoidal groupoid
    (which is a groupoid with an additional symmetric and pseudo-associative operation,
    effectively, a categorical analogue of a commutative monoid).\\

    We move on to developing the cohomological machinery to study coextensions in
    Section~\ref{sec:cohomology_affine}. In particular, we define the classifying
    category \(\fC(M, \cL)\) using the cocycle monoid \(\sZ^1(M,\cL)\), which
    generalises the Grillet complex. Section~\ref{sec:classifying_cat} shows that it
    indeed classifies the precartesian coextensions of \(M\) by \(\cL\), meaning
    \(\Pcoex(M, \cL)\). We also decipher what the Baer sum means in this context and it
    is, as expected, simply the pointwise product.\\

    Our final expedition in the affine case is in the direction of versale (word coming
    from \emph{universal}) coextensions, which are precartesian coextensions of \(M\)
    what have (generally non-unique) maps into every other precartesian coextension of
    \(M\).\\[1em]

    We thus finally move on to Part~\ref{part:functors} which deals with the case of
    monoid functors over posets. To give a quick context as to why we call the first
    part affine in contrast to this, and why we care about coextensions of monoid
    functors over posets, we mention that a monoid scheme is particular type of
    contravariant functor over a (locally lattice) poset. Thus, studying these kinds of
    coextensions are not a ``random'' choices but they correspond exactly to studying
    coextensions of sheaves of monoids over a monoid scheme, for example, if one were
    to study coextensions of quasi-coherent sheaves over a monoid scheme \(X\), such
    monoid functors would be the central objects. This is the motivation behind our
    examination of this topic.\\

    This consists of only a single section, namely Section~\ref{sec:non-affine_pcoex},
    wherein we translate results from the affine to this setting. While most things are
    defined ``locally'', meaning pointwise, this still requires some work, as we have
    to develop a new cohomology theory in Subsection~\ref{sec:global_grillet_pcoex}.
    Showing that the gluing condition holds is also often a fiddly affair that requires
    some verification. There are, however, no great surprises here and one can sum up
    the results of this whole section as saying that ``all affine results translate''.
    For example, \(\Pcoex(\cF, \cL)\) is once again a symmetric monoidal groupoid.\\

    However, unlike our work in \cite{p15}, we do \emph{not} consider the case of
    coextensions of monoid schemes and thus, do not obtain a stack of such groupoids.
    The reason is that, while localisation \(M_\q \to M_\p\) \emph{do} indeed yield
    functors between \(\Pcoex(M_\p, \cL_\p) \to \Pcoex(M_\q \to \cL_\q)\), we can not
    deduce that a coextension of a monoid scheme will be a monoid scheme; we can merely
    say that it will be a monoid functor. Thus, \((\Pcoex(M_\p, \cL_\p))_{\p\in
    \Spec(M)}\) do not classify monoid scheme coextensions anymore and subsequently,
    this approach is not fruitful, and we omit it.

  \section{Preliminaries}
    \label{rem:predecart_cart} \label{def:cartesian_map} \label{pro:unique_scheme}

  \subsection{Preliminaries on the Geometry of Monoids}

      \begin{Conv}[Monoids]
        Monoid are assumed to be commutative throughout this paper, unless explicitly
        stated otherwise. To clarify, we do not assume the existence of an absorbing
        element (zero in the multiplicative notation) and should an element exist which
        happens to absorb everything, monoid homomorphisms are not required to respect
        it.
      \end{Conv}

      We begin by recalling the main results of \cite{p1}, detailing some results about
      prime ideals of commutative monoids. These will be important when we move towards
      monoid schemes. The theory of monoid schemes follows the theory of ring schemes
      very closely, with the exception that we do not have addition.

  \subsubsection{\(\Spec(M)\) and \(M^{sl}\)}

        A subset \(\fa\subseteq M\) of a monoid \(M\) is called an \emph{ideal} if for
        every \(a\in \fa\) and every \(m \in M\), \(am\in \fa\). To put it an other
        way, if \(\fa\) is a sub-\(M\)-set. Ideals are thus allowed to be \emph{empty}.
        An ideal \(\fa\) that is of the form
        \[ (a) \;=\; aM \;:=\; \{am\mid m\in M\} \]
        is called a \emph{principal} ideal generated by \(a\).

        \medskip

        An ideal \(\p\subseteq M\) is called \emph{prime} if \(1 \notin \p\) and
        \(ab\in \p\), \(b\notin \p\) implies \(a\in \p\). An ideal \(\fm\subseteq M\)
        is called \emph{maximal} if \(1 \notin \fm\) and for any chain of inclusions of
        ideals \(\fm \subsetneq \fa \subseteq M\), we have \(\fa = M\).

        It is not hard to see that every monoid \(M\) has a unique maximal ideal,
        namely \(M\setminus M^\times\), the set of all non-invertible elements. Indeed,
        a monoid is a group if and only if it has two ideal, being \((1)\) and
        \(\emptyset\). An other way of saying that is if \(M^\times \subseteq M\), the
        subgroup of invertible elements of \(M\), forms an ideal.

        \medskip

        The set of all prime ideals of a monoid is denoted by \(\Spec(M)\). It carries
        with it a natural topology, called the \emph{Zariski topology}, where open sets
        are generated by sets of the form
        \[ D(f) \;:=\; \{\p \in \Spec(M) \mid f \notin \p\}. \]

        Unlike for rings, the (set) union of ideals is again an ideal and thus, the
        union of prime ideals is again a prime ideal. This makes \(\Spec(M)\) into a
        join-semilattice under union. The least element is the empty set and the set of
        of non-invertible elements (maximal ideal, being the union of all ideals not
        containing \(1\)) is the greatest element. It's dual operation, as promised by
        semilattice theory is given by
        \[ \p\capcup \q \;:=\; \bigcup\limits_{\fr \subseteq \p \cap \q} \fr. \]
        In other words, we take the union of all prime ideas contained in both \(\p\)
        and \(\q\). These operations are continuous in the Zariski topology, making
        \(\Spec(M)\) a topological monoid (or to be more precise, a topological
        semi-lattice).

  \subsubsection{Semilattice theory}

        It becomes evident that semilattices and idempotent monoids are an important
        part of studying monoid theory. For this, let us formally introduce them:

        A \emph{(join) semilattice} is a poset \(L\) with a least element such that
        every pair \(a, b \in L\) has a join (least upper bound) \(a \vee b\). There
        exists a dual notion of meet-semilattice, but as the theory is identical (just
        order-revered), we will not state them. Indeed, we add the following
        convention:
        \begin{Conv}[Semilattice]
          By semilattice we will henceforth mean join-semilattice.
        \end{Conv}
        Semilattices are intimately connected to many parts of mathematics, three of
        which we mention explicitly: monoids, categories and topological spaces.

        \medskip

        Every semilattice is a monoid with \(\vee\) being the operation and the least
        element being the unit. It is clear that \(a\vee a = a\) for every element
        \(a\) of our semilattice. This, the obtained monoid is an idempotent monoid. In
        turn, if \(m^2 = m\) holds for all \(m\in M\) (idempotent monoid), we can
        define an ordering by \(m \leq n\) if and only if \(mn = n\). Under this, the
        unit becomes the least element. This gives us a full equivalence of categories
        \[ \mathsf{Idempotent Monoids} \equiv \mathsf{Semilattices}. \]

        We imbued \(\I\) with the topology that declares \{\(\emptyset\), \(\{1\}\),
        \(\I\)\} as open sets. One checks that \(\I \scong \Spec(\N)\) as topological
        monoids, where \(\N = \{1, t, t^2, \ldots\}\) is the free commutative monoid
        with one generator.

        \medskip

        We have a covariant functor
        \[ \sl: \Mon \to \SL \]
        which assigns \(M/\sim\) to a monoid \(M\), where \(\sim\) is the smallest
        congruence \(\sim\) such that \(m \sim m^2\) for all \(m\). This is the
        universal semilattice quotient of \(M\) and we denote it's image by \(M^{sl}\).
        By Grillet \cite[Theorem~1.2, Ch.~III]{gr}, \(M^{sl} = M/{\sim}\) where \(a
        \sim b\) if and only if there exist \(m, n \geq 1\) and \(u, v \in M\) with
        \(a^m = ub\) and \(b^n = va\). Note that this is the exact relation we have for
        the basis of the Zariski topology, being equivalent to \(D(a) = D(b)\).

        This ties into the link of semilattices with topological spaces. A semilattice
        is in particular a poset and thus, induces a topology where the base of open
        sets is given by subsets of the form \(L_p(P):= \{q\in P \mid p \leq q\}\).
        This induces a topology on both \(\Spec(M)\) and \(M^{sl}\), for a commutative
        monoid \(M\).

        Both \(\Spec(M)\) and \(M^{sl}\) carry a natural topology induced by order,
        where set satisfying the property ``if \(x\in U\) and \(y\leq x\), then \(y\in
        U\)'' are called open. Note, in \(M^{sl}\) we would have to take \(y\geq x\).
        The right-most side carries the subspace topology induced by the product
        topology on \(\prod_{m \in M} \I\).

        We are now in a position to list some of the main results of \cite{p1}.

        \begin{Le}[Reduction Lemma {\cite[Lemma 2.1]{p1}}] \label{lem:spec=hom}
          There are natural homeomorphism of topological monoids
          \[ \Spec(M) \cong \Spec(M^{sl}) \cong \Hom(M, \I). \]
        \end{Le}

        The Reduction Lemma says that the study of \(\Spec(M)\) reduces entirely to the
        study of spectra of semilattices and that the spectrum of a monoid is a
        semilattice itself.

        In the finitely generated case, we also have the following:

        \begin{Le}
          Let \(M\) be a finitely generated monoid. There is an order-reversing
          isomorphism
          \[ \Spec(M)\cong M^{sl}. \]
        \end{Le}

        We also have the following in the non-finitely generated case:

        \begin{Cor}[{\cite[Corollary 2.2]{p1}}]
          If \((M_j)_{j \in J}\) is a filtered direct system of monoids, then the
          canonical map
          \[
            \Spec\!\bigl(\colim_{j} M_j\bigr) \;\longrightarrow\; \lim_{j} \Spec(M_j)
          \]
          is a homeomorphism.
        \end{Cor}

        \begin{proof}
          Since \(\Hom(-, \I)\) converts colimits to limits, the result is immediate
          from Lemma~\ref{lem:spec=hom}.
        \end{proof}

        A first consequence of the Reduction Lemma, which we will use repeatedly, is
        that a surjection of monoids is detected on spectra by the associated
        semilattices.

        \begin{Le}
          Let \(\pi\colon N \to M\) be a surjective monoid homomorphism. Then
          \(\Spec(\pi)\colon \Spec(M)\to\Spec(N)\), \(\p \mapsto \pi^{-1}(\p)\), is an
          injective map of posets. It is an isomorphism of posets if and only if
          \(\Msl[\pi]\colon\Msl[N]\to\Msl[M]\) is injective, in which case
          \(\Msl[\pi]\) is an isomorphism of semilattices.
        \end{Le}

        \begin{proof}
          Since \(\pi\) is surjective, \(\Msl[\pi]\) is surjective, and \(\Spec(\pi) =
          \Hom(\Msl[\pi],\I)\) is injective by Lemma~\ref{lem:spec=hom}. If
          additionally \(\Msl[\pi]\) is injective, it is an isomorphism of semilattices
          and thus \(\Hom(\Msl[\pi],\I)\) is a bijection, which is clearly order
          preserving in both directions. Conversely, if \(\Msl[\pi]\) is not injective,
          then, again by Lemma~\ref{lem:spec=hom}, there are \(n,n'\in N\) with
          \(D(n)\neq D(n')\) but \(D(\pi(n)) = D(\pi(n'))\); any prime of \(N\)
          separating \(n\) from \(n'\) (which exists as \(D(n)\neq D(n')\)) is then not
          of the form \(\pi^{-1}(\p)\), as such preimages are, by construction, unions
          of fibres of \(\pi\).
        \end{proof}

  \subsection{Systems of Commutative Monoids over a Monoid}

      This paper is a continuation of \cite{p15}, where we generalise the theory of
      abelian group extensions to commutative monoid extensions. Subsequently, we will
      need many definitions and constructions that are either verbatim the same, or
      closely related. For the convenience of the reader, there will be some
      repetition, particularly in terms of definitions and constructions. The proofs,
      of course, will not be repeated and the reader is directed to the above cited
      paper should they become relevant. This section is particularly closely related,
      and will thus be kept relatively abridged.

  \subsubsection{Systems of commutative monoids}

        Recall the following definition:

        \begin{De}[{\cite[Definition~def:HC]{p15}}]
          Let \(M\) be a monoid. Denote by \(\bH(M)\) the category whose objects are
          the elements of \(M\) and whose morphisms are given as follows: For \(a,b\in
          M\), a morphism \(a \to b\) is a pair \((r, a)\) with \(ra = b\), written \(a
          \xto{r} ra\). The composition is given by the multiplication law, meaning
          \((a \xto{r} ra) \circ (ra \xto{s} sra) = (a \xto{sr} sra)\).
        \end{De}

        \begin{De}
          A \emph{system of commutative monoids over \(M\)} is a covariant functor
          \(\cL\colon \bH(M) \to \CMon\). Their collection forms a category under
          natural transformations as morphisms, denoted by \(\cHCM(M)\).
        \end{De}

        To spell out the definition, for each \(a\in M\), we have a monoid \(\cL(a)\)
        and for each morphism \((r, a)\in \bH(M)\), a morphism \(r_*\colon \cL(a) \to
        \cL(ra)\), such that \(1_* = \id\) and \((rs)_* = r_* \circ s_*\). A
        \emph{morphism} \(\lambda\colon \cL \to \cL'\), meaning a natural
        transformation, is a collection of morphisms \(\lambda(a)\colon \cL(a) \to
        \cL'(a)\) with commuting squares.

        We compare this to the definition of a system of abelian groups over \(M\),
        which is a covariant functor in the category of abelian groups. It is clear
        that systems of abelian groups embed fully and faithfully in systems of
        commutative monoids.

        \medskip

        It is a basic fact that this category (being a functor category) has all limits
        and colimits, which are computed pointwise.

  \subsubsection{The semidirect product} Let \(\cL\) be a
        system of commutative monoids over a monoid \(M\).

        \begin{De} \label{def:semidirect}
          The \emph{semidirect product} \(\cL\ \rtimes M\) is the monoid \(\{(x,a) \mid
          a \in M,\; x \in \cL(a)\}\) with multiplication
          \[ (x,a)( x', a') \;:=\; ( a_*(x') \cdot a'_*(x)),\, aa'). \]
        \end{De}

        \begin{Pro}
          The category \(\cHCM(M)\) is equivalent to the category of commutative monoid
          objects in the slice category \(\CMon/M\).
        \end{Pro}

        Both results (in the first, we still need to show that it is indeed a monoid)
        is effectively proven in \cite{p15}.

  \subsection{Cartesian and Precartesian Elements}

      In this section, we aim to recall some basic definitions and properties about
      Schreier extensions, (pre-)fibrations, (pre-)cartesian elements and cleavages.
      This is because our theory will be a type of merger/simultaneous generalisation
      of group extension theory and commutative (pre-)cartesian/Schreier extensions.
      More details, including proofs, can be readily found in numerous sources. One
      such example is \cite{p13b}.

      \bigskip

      Let \(\pi\colon N \to M\) be a monoid homomorphism throughout this section.

      \begin{De}
        An element \(n \in N\) is called \emph{\(\pi\)-precartesian} if for every \(z
        \in N\) with \(\pi(z) = \pi(n)\), there exists a unique \(a \in \ker(\pi)\)
        with \(z = na\).
      \end{De}

      We denote the set of all \(\pi\)-precartesian elements by \(\Pcar(\pi)\).

      \begin{Le}
        Let \(n, n' \in \Pcar(\pi)\) with \(\pi(n) = \pi(n')\). There exists a unique
        invertible \(a \in N\) with \(n' = na\).
      \end{Le}

      \begin{De}
        An element \(n \in N\) is called \emph{\(\pi\)-cartesian} if for every \(z \in
        N\) and every factorisation \(\pi(z) = \pi(n) \cdot v\),\, \(v \in M\), there
        exists an element \(y \in N\) with \(z = ny\) and \(\pi(y) = v\). Moreover, if
        \(ny_1 = ny_2\) with \(\pi(y_1) = \pi(y_2)\), we must have \(y_1 = y_2\).
      \end{De}

      The set of \(\pi\)-cartesian elements is denote by \(\Car(\pi)\). One has
      \(N^\times \subseteq \Car(\pi) \subseteq \Pcar(\pi) \subseteq N\). In particular,
      as \(1\in N^\times\) always holds, \(\Pcar(\pi)\) and \(\Car(\pi)\) are always
      non-empty.

      \begin{Le}
        Let \(n_1 \in \Car(\pi)\).
        \begin{enumerate}
          \item If \(n_2 \in \Pcar(\pi)\), then \(n_1 n_2 \in \Pcar(\pi)\).
          \item If \(n_2 \in \Car(\pi)\), then \(n_1 n_2 \in \Car(\pi)\).
        \end{enumerate}
      \end{Le}

      \begin{De}
        A morphism \(\pi\colon N \to M\) is called \emph{fibration} (respectively
        prefibration) if \(\Car(\pi) \to M\) \((\text{respectively}\; \Pcar(\pi) \to
        M)\) is surjective.
      \end{De}

      \begin{Le}
        Let \(\pi\) be a fibration. Then \(\Car(\pi) = \Pcar(\pi)\).
      \end{Le}

      \begin{Le}
        Let \(n \in N\).
        \begin{enumerate}
          \item If \(n\) is invertible, then \(n \in \Car(\pi)\).
          \item If \(\pi(n)\) is invertible, then \(n \in \Car(\pi)\) if and only
                if \(n\) is invertible.
        \end{enumerate}
      \end{Le}

      \begin{De}
        A \emph{cleavage} of a prefibration \(\pi\colon N\to M\) is a function
        \(\mu\colon M \to N\), which assigns to each element \(m\in M\), a
        \(\pi\)-precartesian element \(\mu(m)\in\pi^{-1}(m)\), such that \(\mu(1)=1\).
      \end{De}

      If \(\pi\) is a fibration, then \(\mu(m)\) is automatically \(\pi\)-cartesian,
      see \cite[Section 5]{p13b}.

      \medskip
      \noindent Denote by \(\Clev(\pi)\) the set of all cleaves of \(\pi\). It is not
      empty by the axiom of choice.

      Recall also the definition of a Schreier extension.

      \begin{De}[Schreier extension]
      Let \(A, M\) be monoids with \(A\) additionally commutative and assume there is
      given an action of \(M\) on \(A\), \((a,m)\mapsto am\). A sequence
      \[
      \xymatrix{
        1\ar[r] & A \ar[r]^\iota & N \ar[r]^\sigma & M\ar[r] & 1
      }
    \]
    of monoids and monoid homomorphisms is called a Schreier extension of \(M\) by
    \(A\) if the following conditions hold:
    \begin{enumerate}
      \item \(\sigma\iota(a)=1\) for all \(a\in A\),
      \item \(\sigma\) is a prefibration,
      \item \(\iota\) induces an isomorphism \(A\to \ker(\sigma)\) of
            \(M\)-modules,
      \item \(\iota(a)x=x\iota(\phi_{\sigma(x)}(a))\) for all \(a\in A\) and
            \(x\in N.\)
    \end{enumerate}
    If additionally, \(N\) (and thus \(M\) as well) is commutative, and the action of
    \(M\) on \(A\) is trivial, the extension is called a commutative Schreier extension
    of \(M\) by \(A\).\\

    \noindent If \(\sigma\) is a fibration, the extension is called a regular Schreier
    extension.
    \end{De}

    \part{The affine case}

\section{Definition and first examples}

\subsection{Coextensions}

    We start by recall the definition of group-values coextensions of monoids, as
    introduced in \cite[Definition~3.1]{p15}.

    \begin{De}[Group coextension] \label{def:coext_original}
      Let \(M\) be a monoid and \(\cA\) a system of abelian groups over \(A\). A
      \emph{group coextension} of \(M\) by \(\cA\) is a surjective homomorphism
      \(\pi\colon N \to M\) of monoids, together with actions of \(\cA(m)\) on
      \(\pi^{-1}(m)\) for each \(m\), such that the following two conditions hold:
      \begin{itemize}
        \item[(i)] \emph{(Regularity)} For every \(n, n' \in \pi^{-1}(m)\), there
              exists a unique \(x \in \cA(m)\) with \(n' = x \pt n\).
        \item[(ii)] \emph{(Compatibility)} For all \(x \in \cA(m)\), \(x'\in
              \cA(m')\), \(n \in \pi^{-1}(m)\), \(n' \in \pi^{-1}(n')\), we have
              \[ (x \pt n)(x' \pt n') \;=\; (m'_*(x) + m_*(x')) \pt (nn'). \]
      \end{itemize}
    \end{De}

    The aim of this paper is to generalise the above, and the following is the new
    object of study.

    \begin{De}[Precartesian coextension] \label{def:pcoext}
      Let \(M\) be a monoid and \(\cL\) a system of commutative monoids. A
      \emph{precartesian coextension} of \(M\) by \(\cL\) is a surjective monoid
      homomorphism \(\pi\colon N \to M\), together with actions of \(\cL(m)\) on
      \(\pi^{-1}(m)\) for each \(m \in M\), such that the following conditoins hold:
      \begin{itemize}
        \item[(i)] \emph{(Precartesian condition)} For every \(m \in M\) there
              exists an element \(n_m \in \pi^{-1}(m)\), called \emph{basepoint}, such
              that there is a unique \(x \in \cL(m)\) with \(x \pt n_m = n\) for every
              \(n \in \pi^{-1}(m)\).
        \item[(ii)] \emph{(Compatibility)} For all \(x \in \cL(m)\), \(x' \in
              \cL(m')\), \(n \in \pi^{-1}(m)\), \(n' \in \pi^{-1}(m')\),
              \[
                (x \pt n)(x' \pt n')
                \;=\;
                \bigl(m'_*(x) \cdot m_*(x')\bigr) \pt (nn').
              \]
      \end{itemize}
    \end{De}

    We use the notation \(0 \to \cL \to N \xto{\pi} M \to 0\) to mean a coextension.
    Note also that we used \(\cdot\) as our operation in \(\cL(m)\) instead of \(+\).
    This is because we use multiplication to denote operation in monoids and addition
    to denote the operation in abelian groups, purely out of reasons of convention.

    \begin{Le}
      Let \(0 \to \cL \to N \xto{\pi} M \to 0\) be a precartesian extension and \(m \in
      M\). A point \(n \in \pi^{-1}(m)\) is a basepoint if and only if the map
      \[ \eta_n \colon \cL(m) \to \pi^{-1}(m), \qquad z \mapsto z \pt n, \]
      is bijective.
    \end{Le}

    \begin{Rem}
      Writing \(0 \to \cL \to N \xto{\pi} M \to 0\) is merely notation. In particular,
      \(\cL \to N\) should not be confused with a morphism, as \(\cL\) is not a monoid
      (not even a set for that matter). But the notation is justified in that, when
      \(\cL\) is restricted to a constant system of abelian groups and \(M, N\) are
      abelian groups as well, Proposition~\ref{prop:generalises_group_coext} and our
      work in \cite{p15} show that the above is equivalent to a short-exact sequence in
      the classical sense.
    \end{Rem}

\subsection{Generalising the two main special cases}

    The main goal of this paper is develop a theory that simultaneously generalises
    commutative Schreier extensions and systems of abelian groups developed in
    \cite{p15} (which itself is a generalisation of Grillet's theory). The aim of this
    section is to show that our definition above serves this purpose well.

    \begin{Pro} \label{prop:generalises_group_coext}
      Let \(\cL \in \cHCM(M)\) with \(\cL(m)\) abelian groups for all \(m\in M\). Then
      Definition~\ref{def:pcoext} is equivalent to Definition~\ref{def:coext_original}.
    \end{Pro}

    \begin{proof}
      As the compatibility condition is verbatim the same, we only need to show that
      under this assumptions, the regularity and precartesian conditions agree. We
      start by assuming regularity and choose any \(n_0 \in \pi^{-1}(m)\) as a
      basepoint. By assumption (regularity), we have a unique \(x\) with \(n = x \pt
      n_0\) for every \(n\), implying precartesian.

      On the other hand, let \(n_0\) be a basepoint and choose \(n, n' \in
      \pi^{-1}(m)\) arbitrarily. Let \(x, x'\) be the unique elements of \(\cL(m)\)
      with \(x \pt n_0 = n\) and \(x' \pt n_0 = n'\). By taking \(y = x' \cdot
      x^{-1}\), we get \(y \pt n = (x' \cdot x^{-1}) \pt n = x' \pt (x^{-1} \pt n) = x'
      \pt n_0 = n'\). This \(y\) is unique with that property, since if \(y' \pt n =
      n'\), then \((y'\cdot x) \pt n_0 = y' \pt (x \pt n_0) = y' \pt n = n' = x' \pt
      n_0\). Acting with \(x^{-1}\) on both sides and using commutativity of
      \(\cL(m)\), we get \(y'\pt n_0 = y\pt n_0\). Uniqueness in the precartesian
      assumption now gives us \(y' = y\) and hence, uniqueness in the regularity
      assumption.
    \end{proof}

    \begin{De}
      A system of commutative monoids \(\cL\) on \(M\) is said to be \emph{constant} if
      for every \(m, n\in M\), \(\cL(m) = L\) for a single commutative monoid \(L\) and
      for all \((m, r)\colon m \to mr\), the corresponding map \(r_*\colon L \to L\) is
      the identity.
    \end{De}

    \begin{Pro} \label{prop:generalises_schreier}
      Let \(L\) be a commutative monoid and \(\cL \equiv L\) be the corresponding
      constant system. A precartesian coextension \(0 \to \cL \to N \xto{\pi} M \to 0\)
      is precisely a commutative Schreier extension of \(M\) by \(A\).
    \end{Pro}

\subsection{Basic properties of precartesian coextensions}

    Throughout this section, let
    \[ 0 \to \cL \to N \xto{\pi} M \to 0. \]
    be a given precartesian coextension and write \(n_m\in\pi^{-1}(m)\) for the chosen
    basepoint of each fibre. By Convention~\ref{conv:canonical_basepoint},
    \(n_{1_M}=1_N\).

    \begin{Le} \label{le:pi_respects_action}
      For all \(n \in N\) and \(x \in \cL(\pi(n))\), \(\pi(x \pt n) = \pi(n)\).
    \end{Le}

    \begin{Le} \label{le:basepoints_invertibly_related}
      Let \(n_0, n_0'\) be two basepoints of \(\pi^{-1}(m)\). There is a unique
      invertible \(u \in \cL(m)\) with \(u \pt n_0 = n_0'\). In particular, the set of
      basepoints of \(\pi^{-1}(m)\) is a \(\cL(m)^\times\)-torsor.
    \end{Le}

    \begin{proof}
      By assumption (precartesian), there exist unique \(u, v \in \cL(m)\) with \(u \pt
      n_0 = n_0'\) and \(v \pt n_0' = n_0\). As such,
      \[ (v\cdot u) \pt n_0 \;=\; v \pt (u \pt n_0) \;=\; v \pt n_0' \;=\; n_0. \]
      Since we also have \(e_m \pt n_0 = n_0\), the uniqueness in the precartesian
      condition (applied at the basepoint \(n_0\)) now implies \(v\cdot u = e_m\). By
      symmetry, we als get
      \[ (u\cdot v) \pt n_0' \;=\; u \pt (v \pt n_0') \;=\; u \pt n_0 \;=\; n_0', \]
      and uniqueness at the basepoint \(n_0'\) gives \(u\cdot v = e_m\). Hence \(u\) is
      invertible with inverse \(v\).

      Uniqueness of \(u\) with \(u \pt n_0 = n_0'\) is immediate from the precartesian
      condition at \(n_0\).

      It is also readily seen that if \(n_0\) is a basepoint and \(u\in \cL(m)^\times\)
      invertible, \(u\pt n_0\) is likewise a basepoint.
    \end{proof}

    Clearly, if \(1_N\in \pi^{-1}(m)\) (meaning, \(m=1\)), we can always take \(1_N\)
    to be a basepoint. The above lemma tells us that we can do so more or less
    canonically, and henceforth, we shall do so:

    \begin{Conv} \label{conv:canonical_basepoint}
      We take \(1_N\) to be the canonical basepoint of \(\pi^{-1}(1_M)\).
    \end{Conv}

    \begin{Le}
      \begin{enumerate}
        \item The map \(f\colon \cL(1_M) \to N\), given by \(f(x) = x \pt 1_N\), is an
              injective homomorphism.
        \item If \(\cL(1_M)\) is a group, then \(n \in N^\times\) if and only if
              \(\pi(n) \in M^\times\).
      \end{enumerate}
    \end{Le}

    \begin{proof}
      (1) We use the compatibility condition with \(n = n' = 1_N\) to get
      \begin{eqnarray*}
        f(x) f(x') & = & (x \pt 1_N)\cdot (x' \pt 1_N)                        \\
                   & = & ((1_M)_*(x)) \cdot ((1_M)_*(x')) \pt (1_N \cdot 1_N) \\
                   & = & (x \cdot x') \pt (1_N)                               \\
                   & = & f(xy).
      \end{eqnarray*}
      This implies that \(f\) is a monoid homomorphism.

      To see injectivity, assume \(f(x) = f(x')\), meaning \(x \pt 1_N = x' \pt 1_N\).
      The uniqueness property of the precartesian condition then implies \(x = x'\).

      \medskip

      (2) Let \(\pi(n) \in M^\times\) be invertible and \(n'\in
      \pi^{-1}(\pi(n)^{-1})\). It follows that \(nn'\in \pi^{-1}(1_M)\), implying the
      existence of a unique \(x \in \cL(1_M)\) with \(nn' = x \pt 1_N = f(x)\). But as
      \(\cL(1_M)\) is a group, every element (including \(x\)) is invertible, and a
      morphism maps invertibles to invertibles. This makes \(f(x)\) and as such \(n\)
      invertible.
    \end{proof}

\subsection{Morphisms}

    \begin{De} \label{def:morphism_precoext}
      A \emph{precartesian morphism}, or simply \emph{morphism} \(\cE \to \cE'\) of
      precartesian coextensions over (possibly different) bases is a triple \((\lambda,
      h, \varphi)\),
      \[
        \xymatrix{
          \cE \;=\; \ar@<-1ex>[d] & 0 \ar[r]                 & \cL \ar[r]\ar@{=>}[d]^{\lambda} & N
          \ar[r]^{\pi}\ar[d]^{h}  & M \ar[r]\ar[d]^{\varphi} & 0 \\
          \phantom{'}\cE' \;=\;   & 0 \ar[r]                 & \cL' \ar[r]                     & N' \ar[r]^{\pi'} & M' \ar[r]
                                  & 0
        }
      \]
      where \(h, \varphi\) are monoid homomorphisms and \(\lambda\colon \cL \to
      \varphi^*\cL'\) is a morphism in \(\cHCM(M)\) (meaning natural transformation),
      satisfying
      \[
        \varphi \circ \pi \;=\; \pi' \circ h \,\qqtext{and}\, (h(x \pt n) \;=\;
        \lambda(x) \pt h(n)
      \]
      for all \(n\in N\), \(x \in \cL(m)\) for every \(m\in M\). Moreover, \(h\) sends
      basepoints to basepoints, meaning for every \(m \in M\) and every basepoint \(n_0
      \in \pi^{-1}(m)\), \(h(n_0)\) is a basepoint of \(\pi'^{-1}(\varphi(m))\).
    \end{De}

    We will denote the category of precartesian coextensions of \(M\) by \(\Pcoex(M)\),
    with morphisms as in Definition~\ref{def:morphism_precoext}. In the case when we
    have a fixed system \(\cL\), we will use the notation \(\Pcoex(M, \cL)\). The
    morphisms here are \((\id_\cL, h, \id_M)\).

\section{First results} \label{sec:first-results}

\subsection{Split precartesian coextensions} \label{sec:split-coextension}

    \begin{De}
      A precartesian coextension \emph{splits} (or is \emph{split}) if there exists a
      cleavage monoid homomorphism \(\mu\colon M \to N\) with \(\pi \circ \mu =
      \id_M\).
    \end{De}

    To clarify, by cleavage monoid homomorphism we mean a cleavage that is additionally
    a monoid homomorphism.

    The semidirect product \(\cL \rtimes M\), with \(\pi(x,a) = a\) and \(z \pt (x,a)
    := (z \cdot x,a) \) (see Definition~\ref{def:semidirect}), is split precartesian
    coextension, with basepoint of the fibre over \(a\) being \(( e_a,a)\) and the
    action given by \(z \pt (e_a,a) = ( z,a)\). This is the canonical such coextension.

    \begin{De}
      A \emph{derivation} \(\partial\colon M \to \cL\) is a ``function'' with
      \(\partial(m) \in \cL(m)\) satisfying
      \[ \partial(mm') \;=\; m'_*(\partial(m)) \cdot m_*(\partial(m')) \]
      for all \(m, m' \in M\).
    \end{De}

    \begin{De}
      Let \(\pi\colon M\to N\) be a monoid homomorphism and \(\partial\colon M \to
      \cL\) a derivation. We call \(\partial\) a \emph{cartesian derivation} or simply
      \emph{cartesian} if it maps basepoint elements to invertible elements.
    \end{De}

    See Remark~\ref{rem:predecart_cart} for why we call it cartesian and not
    precartesian.

    \medskip

    The following lemma, in conjunction with Definition~\ref{def:cartesian_map} makes
    the definition of cartesian derivation clearer:

    \begin{Le} \label{le:semidirect_cartesian=invertible}
      Let \(\pi\colon \cL \rtimes M \to M\) be the split precartesian coextension. An
      element \((x,m)\in\pi^{-1}(m)\) is a basepoint of \(\pi^{-1}(m)\), in the sense
      of Definition~\ref{def:pcoext}(i), if and only if \(x\in\cL(m)^\times\).
    \end{Le}

    \begin{proof}
      Let \(\mu_0\colon M\to \cL \rtimes M \), \(\mu_0(m):=(e_m,m)\) be the canonical
      cleavage of \(\pi\). By Lemma~\ref{le:basepoints_invertibly_related}, the set of
      basepoints of \(\pi^{-1}(m)\) is precisely the \(\cL(m)^\times\)-orbit of
      \(\mu_0(m)\) under the canonical action \(u \pt (e_m,m) = ( u\cdot e_m, \, m) =
      (u,m)\). Hence \((x,m)\) is a basepoint if and only if \(x=u\) for some
      \(u\in\cL(m)^\times\).
    \end{proof}

    \begin{Rem}
      Henceforth, calling an element \(n\) of a coextension \(\pi\colon N\to M\)
      \emph{precartesian} means that \(n\) is a basepoint of its fibre
      \(\pi^{-1}(\pi(n))\) in the sense of Definition~\ref{def:pcoext}(i). Note also
      that we simply say precartesian and not \(\pi\)-precartesian.
    \end{Rem}

    We mention that the trivial map \(m \mapsto e_m\), where \(e_m\) is the unit of the
    group \(\cL(m)\) is always a derivation.

    \begin{Le} \label{le:split}
      The following are equivalent:
      \begin{enumerate}
        \item A precartesian coextension is split.
        \item There is an isomorphism of precartesian coextensions \((\id_\cL, h,
              \id_M)\) with \(h\colon N \xto{\sim} \cL \rtimes M\)
        \item There is a cartesian derivation \(\partial\colon N \to \pi^*\cL\) with
              \(\partial(x \pt n) = x \cdot \partial(n)\).
      \end{enumerate}
    \end{Le}

    \begin{proof}
      \((1)\Rightarrow(3)\). Let \(\mu\colon M \to N\) be the splitting. By assumption
      (cleavage), we know that \(\mu(\pi(n))\) is a basepoint and hence, there exists
      an element, formally denoted \(\partial(n) \in \cL(\pi(n))\) such that
      \[ \partial(n)\pt\mu(\pi(n)) \;=\; n. \]
      Define \(\partial\colon N \to \pi^*\cL\) through \(n \mapsto \partial(n)\). We
      have to check three things: That \(\partial\) is a derivation, that \(\partial\)
      satisfies respects the action and that \(\partial\) is cartesian.

      \medskip\noindent\emph{Derivation:} Let \(n,n'\in N\), and \(m=\pi(n)\),
      \(m'=\pi(n')\). Since \(\mu\) is a monoid homomorphism,
      \(\mu(mm')=\mu(m)\mu(m')\). The compatibility condition now gives us
      \begin{align*}
        nn' & = \bigl(\partial(n)\pt\mu(m)\bigr)\bigl(\partial(n')\pt\mu(m')\bigr)               \\
            & = \bigl(m'_*(\partial(n))\cdot m_*(\partial(n'))\bigr)\pt\bigl(\mu(m)\mu(m')\bigr) \\
            & = \bigl(m'_*(\partial(n))\cdot m_*(\partial(n'))\bigr)\pt\mu(mm').
      \end{align*}
      On the other hand, by definition, \(\partial(nn') \in \cL(mm')\) is the unique
      element of satisfying \(\partial(nn')\pt\mu(mm')=nn'\). By uniqueness,
      \[ \partial(nn') \;=\; m'_*(\partial(n))\cdot m_*(\partial(n')). \]
      This is the derivation identity.

      \medskip\noindent\emph{Action:} By definition of \(\partial\) and
      Lemma~\ref{le:pi_respects_action}, \(\partial(x\bullet n)\) is the unique element
      satisfying \(\partial(x\pt n)\pt\mu(\pi(n)) = x\pt n\). On the other hand,
      \[
        (x\cdot\partial(n))\pt\mu(m) \;=\; x\pt\bigl(\partial(n)\pt\mu(m)\bigr) \;=\;
        x\pt n.
      \]
      Equating these gives us
      \[ \partial(x\pt n)\pt\mu(\pi(n)) \;=\; (x\cdot\partial(n))\pt\mu(m). \]
      Uniqueness now implies \(\partial(x\pt n)=x\cdot\partial(n)\), as desired.

      Lastly, we have to verify that it is cartesian. But this is just
      Lemma~\ref{le:basepoints_invertibly_related}, since if we start out with a
      cartesian element and return via \(\mu\circ \pi\) to a new basepoint, the
      connecting element \(\partial(n)\in \cL(n)\) must be a unit.

      \medskip\noindent\((3)\Rightarrow(2)\):
      Define the map
      \[
        h\colon N\longrightarrow \cL \rtimes M,
        \qquad h(n) \;:=\; (\partial(n) ,\,\pi(n)).
      \]
      We claim that this is the desired isomorphism. We start by showing that it is a
      homomorphism. Let \(n,n'\in N\) and \(m=\pi(n)\), \(m'=\pi(n')\). We have
      \begin{align*}
        h(n)\cdot h(n') & = \bigl(\partial(n),\pi(n)\bigr)\cdot\bigl(\partial(n'), \pi(n')\bigr) \\
                        & = \bigl(m_*(\partial(n'))\cdot m'_*(\partial(n), \pi(n)\pi(n')\bigr)   \\
                        & = \bigl(\partial(nn'), \pi(nn')\bigr)                                  \\
                        & = h(nn'),
      \end{align*}
      where the second equality is the product identity of the semi-direct product and
      the third equality follows from \(\partial\) begin a derivation.

      We also have \(h(1) = ( \partial(1), \pi(1)) = ( \partial(1 \cdot 1), 1)\). Using
      the derivation identity with \(n=n'=1_N\) gives us \(\partial(1 \cdot
      1)=\partial(1)^2\). Since \(1_N\) is the canonical basepoint, hence precartesian,
      of \(\pi^{-1}(1_M)\), and \(\partial\) is cartesian, \(\partial(1)\) is
      invertible in \(\cL(1_M)\). Multiplying \(\partial(1)=\partial(1)^2\) by
      \(\partial(1)^{-1}\) gives \(\partial(1) = 1\), and we are done.

      To see that it is an isomorphism, we will show that it is bijective on each
      fibre. We have the bijection \(\pi^{-1}(m) \simeq \cL(m)\) due to the
      precartesianness, given by \(\pi^{-1}(m)\ni n \mapsto x\) for the unique \(x\)
      satisfying \(n = x \bullet n_m\). Hence, we only need to show that the restricted
      \(h|_{\pi^{-1}(a)}: \cL(m) \to \cL(m) \times \{m\}\) is a bijection under this
      identification. But this is given by \(x\mapsto x \cdot \partial(n_m)\). Since
      \(\partial(n_m)\) is invertible, we are done.

      Lastly, we have to show that \((\id_\cL,h,\id_M)\) is a morphism of precartesian
      coextensions in the sense of Definition~\ref{def:morphism_precoext}, meaning in
      particular that \(h\) respects basepoints: indeed \(h\) sends the basepoint
      \(\mu(m)\) of \(\pi^{-1}(m)\) to \(( \partial(\mu(m)),m) = ( e_m,m)\), the
      basepoint of \( \cL \rtimes M\) over \(m\), using \(\partial(\mu(m))=e_m\) (which
      follows since \(\mu(m) = e_m \pt \mu(m)\) and uniqueness of \(\partial\)). This
      can be readily verified in full and we skip the remaining details.

      \medskip\noindent\((2)\Rightarrow(1)\).
      The semidirect product \(\cL \rtimes M\) admits the canonical cleavage monoid
      homomorphism
      \[ \mu_0\colon M\longrightarrow \cL\rtimes M,\qquad\mu_0(m) \;:=\; (e_m,m). \]
      It is readily checked that this is a monoid homomorphism. Defining \(\mu :=
      h^{-1}\circ\mu_0\colon M\to N\), we claim that we get our desired splitting. That
      this is a splitting (monoid homomorphism + \(\pi\circ\mu=\pi\circ
      h^{-1}\circ\mu_0\)) is clear. The only thing to verify is that it is a cleavage.
      By Lemma~\ref{le:semidirect_cartesian=invertible}, \(\mu_0\) is clearly a
      cleavage, since the unit is always invertible. But isomorphisms respect all
      algebraic structures and in particular, precartesianness. Hence, \(\mu\) is also
      a cleavage and we are done.
    \end{proof}

\subsection{Basepoints under morphisms}

    We can reformulate Lemma~\ref{le:basepoints_invertibly_related} to strengthen
    Lemma~\ref{le:semidirect_cartesian=invertible} as follows:

    \begin{Le} \label{le:cartesian_criterion}
      Let \( 0 \to \cL \to N \xto{\pi} M \to 0 \) be a precartesian coextension and
      \(\kappa: M\to N\) a cleavage of \(\pi\). For any \(m\in M\) the set of
      basepoints in \(\pi^{-1}(m)\) is equal to
      \[ \{u \pt \kappa(m) \mid u \in \cL(m)^\times \} \]
      In particular, \(\Pcar(\cL \rtimes M) = \cL^\times\) as \(M\)-graded sets and,
      more generally, for a coextension
      \[ 0 \to \cL \to {}_fM \xto{\pi} M \to 0, \]
      the set of basepoints in \(\pi^{-1}(m)\) equals to
      \[ \{ (x,a) \mid x \in \cL(a)^\times \}. \]
    \end{Le}

    Observe also that if \(\varphi\) is injective and we have chosen a cleavage
    \(\kappa\) in the diagram below, we can choose a cleave \(\kappa'\) such that the
    diagram commutes
    \[
      \xymatrix{
        0 \ar[r]                                       & \cL \ar[r]\ar@{=>}[d]^{\lambda} & N \ar[r]^{\pi}\ar[d]^{h} & M
        \ar[r] \ar[d]^{\varphi} \ar@/_1em/[l]_{\kappa} & 0 \\
        0 \ar[r]                                       & \cL' \ar[r]                     & N' \ar[r]^{\pi'}         & M' \ar[r] \ar@/^1em/[l]^{\kappa'} &
        0.
      }
    \]
    Hence
    \begin{equation*} \label{eq:precartesian_kappa}
      h(\kappa(m)) = \kappa'(\varphi(m))
    \end{equation*}
    for all \(m\in M\).

    \begin{Le}[Short Five Lemma] \label{lem:short5_precart}
      Let \((\lambda, h, \varphi) \colon \cE \to \cE'\) be a morphism of precartesian
      coextensions
      \[
        \xymatrix{
          0 \ar[r]                & \cL \ar[r]\ar@{=>}[d]^{\lambda} & N \ar[r]^{\pi}\ar[d]^{h} & M
          \ar[r] \ar[d]^{\varphi} & 0 \\
          0 \ar[r]                & \cL' \ar[r]                     & N' \ar[r]^{\pi'}         & M' \ar[r] &
          0.
        }
      \]
      \begin{enumerate}
        \item If \(\varphi\) and \(\lambda\) are injective, then \(h\) is injective.
        \item If \(\varphi\) and \(\lambda\) are isomorphisms, then \(h\) is an
              isomorphism.
      \end{enumerate}
    \end{Le}

    \begin{proof}
      (1) Assume \(h(n) = h(n')\). Then \(\pi'(h(n)) = \pi'(h(n')) \;\Leftrightarrow\;
      \varphi(\pi(n)) = \varphi(\pi(n'))\) and thus, the injectivity of \(\varphi\)
      gives us \(\pi(n) = \pi(n')\). Let us call this element \(m\in M\) and denote the
      basepoint of \(\pi^{-1}(m)\) by \(n_0\). This makes \(h(n_0)\) a basepoint of
      \(\cE'\), since \(h\) respects basepoints by
      Definition~\ref{def:morphism_precoext}. Let \(x, y\) be the unique elements
      satisfying \(x \pt n_0 = n\), \(y \pt n_0 = n'\). It follows that
      \[ \lambda(x) \pt h(n_0) \;=\; h(n) \;=\; h(n') \;=\; \lambda(y) \pt h(n_0), \]
      and thus, \(\lambda(x) = \lambda(y)\) by uniqueness (using that \(h(n_0)\) is a
      basepoint). It follows that \(x = y\) by the injectivity of \(\lambda\), making
      \(h\) injective.

      (2) Take an element \(n'\) in \(N'\) and denote \(m:= \pi'(n') m'\). Choose a
      basepoint \(n_0'\) in \(\pi'^{-1}(m')\). Then \(n'=x'n'_0\) for a uniquely
      defined \(x'\in \cL'(m')\). Since \(\varphi\) is surjective, we can write
      \(m'=\varphi(m)\) for some \(m\in M\). Choose a basepoint \(n_0\) in
      \(\pi^{-1}(m)\). By assumption \(h\) sends basepoints to basepoints. It follows
      that \(h(n_0)\) is also a basepoint in \(\pi'^{-1}(m')\).

      As such, \(h(n_0)=y'n_0'\) for an invertible element \(y'\in \cL'(m')\) and thus,
      \(n'=x'(y')^{-1}h(n_0)\). Since
      \[ \lambda(m):\cL(m)\to \cL'(\varphi(m)) \;=\; \cL'(m') \]
      is surjective, there exists an element \(x\in \cL(m)\) such that
      \(\lambda(x)=x'(y')^{-1}\). We obtain
      \[ h(xn_0) \;=\; \lambda(x)h(n_0) \;=\; x'(y')^{-1}y'n_0'=x'n_0'=n' \]
      and surjectivity of \(h\) follows.
    \end{proof}

    As a direct corollary of Lemma~\ref{lem:short5_precart}, we have the following
    corollary in the notation of Section~\ref{sec:pcoex}.

    \begin{Cor} \label{cor:Pcoex_groupoid}
      \(\Pcoex(M,\cL)\) is a groupoid.
    \end{Cor}

\section{Pushforwards, and Pullbacks}\label{sec:pushforward_pullback}

\subsection{Pushforward along a morphism of systems} \label{sec:pushforward}

    Let \(\alpha\colon \cL \to \cL'\) be a morphism in \(\cHCM(M)\) and let \(\cE = (0
    \to \cL \to N \xto{\pi} M \to 0) \in \Pcoex(M)\) be a precartesian coextension. We
    wish to construct a coextension \(\alpha_*\cE \in \Pcoex(M, \cL')\) fitting into a
    commutative diagram of morphisms of precartesian coextensions
    \[
      \xymatrix{
        0 \ar[r] & \cL \ar[r] \ar@{=>}[d]_{\alpha} & N \ar[r]^{\pi} \ar@{..>}[d]^{\rho}
                 & M \ar[r] \ar[d]^{\id}           & 0 \\
        0 \ar[r] & \cL' \ar@{..>}[r]               & K \ar@{..>}[r]_{\sigma}
                 & M \ar[r]                        & 0.
      }
    \]
    We fix a cleavage \(\kappa\colon M \to N\) throughout this section. That is to say,
    for each \(m \in M\), we have a basepoint \(\kappa(m) \in \pi^{-1}(m)\), which
    exists by the precartesian condition. We also write
    \[
      \xi\colon N \to \cL \,\qtext{with}\, \xi(n) \;\in\; \cL(\pi(n)),
    \]
    where \(\xi(n)\) is the unique element provided by the precartesian condition
    satisfying \(\xi(n) \pt \kappa(\pi(n)) = n\). Note that \(\kappa\) is not a
    splitting in general, as it need not be a homomorphism, and likewise, \(\xi\) is
    not generally a derivation.

    \begin{Const}[Pushforward] \label{const:pushforward}
      Consider the set of pairs
      \[
        \widetilde K \;:=\; \bigl\{(y,n) \;\bigm|\; n \in N,\; y \in
        \cL'(\pi(n))\bigr\}
      \]
      and define the relation \(\sim\) on \(\widetilde K\) under which \((y,n) \sim
      (y',n')\) if and only if
      \begin{enumerate}
        \item[(i)] \(\pi(n) = \pi(n')\), and
        \item[(ii)] \(y \cdot \alpha_{\pi(n)}(\xi(n)) \;=\; y' \cdot
              \alpha_{\pi(n)}(\xi(n'))\) in \(\cL'(\pi(n))\).
      \end{enumerate}
      Write \([y,n]\) for the \(\sim\)-class of \((y,n)\), define the operation
      \[
        [y,n] \cdot [z,n'] \;:=\; \bigl[\pi(n')_*(y) \cdot \pi(n)_*(z),\; nn'\bigr]
      \]
      on \(K := \widetilde K / {\sim}\), and define \(\sigma\colon K \to M\) by
      \(\sigma[y,n] = \pi(n)\). For \(u \in \cL'(m)\), define the action on
      \(\sigma^{-1}(m)\) through \(u \pt [y,n] = [u \cdot y,\, n]\).
    \end{Const}

    \begin{Le} \label{le:pushforward_welldef}
      The relation \(\sim\) is an equivalence relation and the operation on \(K\) is
      well-defined. The triple \(\alpha_*\cE := (0 \to \cL' \to K \xto{\sigma} M \to
      0)\) is a precartesian coextension of \(M\) by \(\cL'\), with basepoint
      \(\kappa'(m) := [e_m, \kappa(m)]\) of \(\sigma^{-1}(m)\), where \(e_m \in
      \cL'(m)\) denotes the identity element.
    \end{Le}

    \begin{proof}
      That \(\sim\) is an equivalence relation is immediate. To see that our operation
      is well-defined, let \((y_1, n_1) \sim (y_2, n_2)\) with \(m = \pi(n_1) =
      \pi(n_2)\) and \((y'_1, n'_1) \sim (y'_2, n'_2)\) with \(m' = \pi(n'_1) =
      \pi(n_2')\). We must show that
      \[
        \bigl(m'_*(y_1) \cdot m_*(y'_1),\; n_1 n_1'\bigr)
        \;\sim\;
        \bigl(m'_*(y_2) \cdot m_*(y'_2),\; n_2 n_2'\bigr).
      \]
      For this, we have to show two equalities:
      \begin{itemize}
        \item[(i)] \(\pi(n_1n_1') = \pi(n_2n_2')\) and
        \item[(ii)] \(m'_*(y_1)\cdot m_*(y'_1)\cdot \alpha_{\pi(n_1n'_1)}(\xi(n_1n'_1))
              = m'_*(y_2)\cdot m_*(y'_2)\cdot \alpha_{\pi(n_2n'_2)}(\xi(n_2n'_2))\)
      \end{itemize}
      The first one is trivial, as \(\pi\) is a homomorphism. For the second one, we
      need a bit more arguing. Let \(i\in \{1, 2\}\). We start by understanding
      \(\xi(n_in'_i)\). Using the compatibility condition for \(n_i =
      \xi(n_i)\pt\kappa(m)\) and \(n_i' = \xi(n_i')\pt\kappa(m')\) gives us
      \begin{eqnarray*}
        n_i n'_i & = & \left(\xi(n_i)\pt\kappa(m)\right)\left(\xi(n'_i)\pt\kappa(m')\right) \\
                 & = & \left(m'_*(\xi(n_i))\cdot m_*(\xi(n'_i))\right)\pt\left(\kappa(m) \cdot \kappa(m')\right).
      \end{eqnarray*}
      Using this, we get
      \begin{eqnarray*}
        \xi(n_i n_i') \pt \kappa(mm') & = & n_1n'_i                                                                                   \\
                                      & = & \left(m'_*(\xi(n_i))\cdot m_*(\xi(n_i'))\right)\pt\left(\kappa(m) \cdot \kappa(m')\right) \\
                                      & = & \left(m'_*(\xi(n_i))\cdot m_*(\xi(n_i'))\cdot c\right)\pt\kappa(mm'),
      \end{eqnarray*}
      where the last equality of the associativity condition and the action property,
      with
      \[ c \;:=\; \xi(\kappa(m)\kappa(m')) \in \cL(mm') \]
      being shorthand for the unique element satisfying \(c \pt \kappa(mm') \;=\;
      \kappa(m) \cdot \kappa(m')\). The uniqueness of \(\xi(n_i n_i')\) now implies
      \begin{equation} \label{eq:xi_product}
        \xi(n_i n_i') \;=\; m'_*(\xi(n_i))\cdot m_*(\xi(n_i'))\cdot c.
      \end{equation}
      Using the fact that \(\alpha\) is a natural isomorphism (and thus commutes with
      \(m_*, m'_*\)), we can do the following simplification of Equation~(ii) (of the
      well-definedness) using Equation~\eqref{eq:xi_product}:
      \begin{eqnarray*}
        m'_*(y_i)\cdot m_*(y'_i)\cdot \alpha_{\pi(n_in'_i)}(\xi(n_in'_i)) & = & m'_*(y_i)\cdot m_*(y'_i)\cdot \alpha_{mm'}(m'_*(\xi(n_i))\cdot m_*(\xi(n_i'))\cdot c)                        \\
                                                                          & = & m'_*(y_i)\cdot m_*(y'_i)\cdot m'_*(\alpha_m(\xi(n_i)))\cdot m_*(\alpha_{m'}(\xi(n_i')))\cdot \alpha_{mm'}(c) \\
                                                                          & = & m'_*(y_i\cdot \alpha_m(\xi(n_i)))\cdot m_*(y'_i \cdot \alpha_{m'}(\xi(n_i')))\cdot \alpha_{mm'}(c)
      \end{eqnarray*}
      But as \(y_1\cdot \alpha_m(\xi(n_1)) = y_2\cdot \alpha_m(\xi(n_2))\) (since
      \((y_1, n_1) \sim (y_2, n_2)\)), we are done.

      It is clear that the structure is commutative, as it is defined symmetrically.

      To see associativity, let \(m_1 = \pi(n_1)\), \(m_2 = \pi(n_2)\) and \(m_3 =
      \pi(n_3)\). Using functoriality \((m_1m_2)_* = m_{1*}\circ m_{2*}\), we have
      \begin{eqnarray*}
        \bigl[[y_1, n_1] \cdot [y_2, n_2]\bigr] \cdot [y_3, n_3] & = & [m_{2*}(y_1) \cdot m_{1*}(y_2), n_1 n_2] \cdot [y_3, n_3] \\
                                                                 & = & [(m_2m_3)_*(y_1) \cdot (m_1m_3)_*(y_2) \cdot (m_1m_2)_*(y_3), n_1 n_2 n_3]
      \end{eqnarray*}
      and on the one hand
      \begin{eqnarray*}
        [y_1, n_1] \cdot \bigl[[y_2, n_2]\cdot [y_3, n_3]\bigr] & = & [y_1, n_1] \cdot [m_{3*}(y_2) \cdot m_{2*}(y_3), n_2 n_3] \\
                                                                & = & [(m_2m_3)_*(y_1) \cdot (m_1m_3)_*(y_2) \cdot (m_1m_2)_*(y_3), n_1 n_2 n_3],
      \end{eqnarray*}
      giving us associativity.

      The identity is played by \([e_{1_M}, 1_N]\).

      \medskip It is obvious that \(\sigma\) is a well-defined surjective monoid
      homomorphism, and likewise, that the action is well-defined.

      We verify that the compatibility condition holds:

      Let \(u \in \cL'(m)\), \(u' \in \cL'(m')\), \([y,n]\in\sigma^{-1}(m)\),
      \([y',n']\in\sigma^{-1}(m')\). We have
      \begin{eqnarray*}
        (u\pt[y,n])\cdot(u'\pt[y',n']) & = & [u\cdot y,\,n]\cdot[u'\cdot y',\,n']                    \\
                                       & = & [m'_*(u\cdot y)\cdot m_*(u'\cdot y'),\; nn']            \\
                                       & = & [m'_*(u)\cdot m'_*(y)\cdot m_*(u')\cdot m_*(y'),\; nn'] \\
                                       & = & (m'_*(u)\cdot m_*(u'))\pt[m'_*(y)\cdot m_*(y'),\; nn']  \\
                                       & = & (m'_*(u)\cdot m_*(u'))\pt\bigl([y,n]\cdot[y',n']\bigr),
      \end{eqnarray*}
      which is Definition~\ref{def:pcoext}(ii) for \(\cL'\).

      \medskip

      Lastly, we show that \(\kappa'(m) = [e_m, \kappa(m)]\) serves as a basepoint of
      \(\sigma^{-1}(m)\). For this, we have to show that for every \([y,n]\in
      \sigma^{-1}(m)\), there exists a unique \(u\in\cL'(m)\) with \(u\pt\kappa'(m) =
      [y,n]\).

      We have
      \[
        u\pt[e_m,\kappa(m)] \;=\; [u\cdot e_m,\,\kappa(m)] \;=\; [u,\,\kappa(m)] \] and
        thus, for \([u,\kappa(m)] = [y,n]\) to hold, we must have \[ u\cdot \alpha_m
        (\xi(\kappa(m))) \;=\; y\cdot \alpha_m(\xi(n)) \] as \(\pi(\kappa(m)) =
        \pi(n)\) always holds. As \(\xi(\kappa(m)) = e_m\), the left hand side
        simplifies to \(u \cdot \alpha_m(e_m) = u\) and thus, \[ u \;=\;
        y\cdot\alpha_m(\xi(n))
      \]
      is the unique solution.
    \end{proof}

    This proves that \(0 \to \cL' \to K \xto{\sigma} M \to 0\) is a coextension. Next,
    we ave to show that there exists a morphism between this and the original
    coextension.

    \begin{Pro} \label{pro:pushforward_morphism}
      The tripple \((\alpha, \rho, \id_M)\colon \cE \to \alpha_*\cE\) is a precartesian
      homomorphism, where \(\rho\colon N \to K\) is given by \(n \mapsto [e_{\pi(n)},\,
      n]\).
    \end{Pro}

    \begin{proof}
      We start by showing that \(\rho\) is a monoid homomorphism. Indeed, for \(n,
      n'\in N\) with \(\pi(n)=m\), \(\pi(n')=m'\) we have
      \[
        \rho(n)\cdot\rho(n') \;=\; [e_m,n]\cdot[e_{m'},n'] \;=\; [m'_*(e_m)\cdot
        m_*(e_{m'}),\,nn'] \;=\; [e_{mm'},\,nn'] \;=\; \rho(nn').
      \]
      We also have \(\sigma\circ\rho=\pi\) as \(\sigma[e_m,n]=\pi(n)\). For the
      equivariance with the action, let \(x\in\cL(m)\) and \(n\in\pi^{-1}(m)\). We have
      \((x\cdot\xi(n))\pt\kappa(m) = x\pt(\xi(n)\pt\kappa(m)) = x\pt n\). Uniqueness of
      the values associated by \(\xi\) now gives us \(\xi(x\pt n) = x\cdot\xi(n)\). We
      can follow
      \[
        e_m\cdot\alpha_m(\xi(x\pt n)) \;=\; e_m\cdot\alpha_m(x\cdot\xi(n)) \;=\;
        \alpha_m(x)\cdot\alpha_m(\xi(n)) \;=\; \alpha_m(x)\cdot\alpha_m(\xi(n))
      \]
      and thus, \([e_m, x\pt n] = [\alpha_m(x), n]\). But this means \(\rho(x\pt n) =
      \alpha_m(x)\pt\rho(n)\) as required.

      Lastly, to see that basepoints are respected, observe that \(\rho(\kappa(m)) =
      [e_m, \kappa(m)] = \kappa'(m)\).
    \end{proof}

\subsection{Pullback along a monoid homomorphism} \label{sec:pullback}

    Let \(\varphi\colon M' \to M\) be a monoid homomorphism and \(\cE = (0 \to \cL \to
    N \xto{\pi} M \to 0) \in \Pcoex(M, \cL)\) a precartesian coextension with cleavage
    \(\kappa\). Our aim is to construct the pullback \(\varphi^*\cE \in \Pcoex(M',
    \varphi^*\cL)\) fitting into a commutative diagram
    \begin{equation} \label{eq:pullback_diagramme}
      \xymatrix{
      0 \ar[r] & \varphi^*\cL \ar[r] \ar@{=>}[d] & \varphi^*N \ar[r]^{\sigma} \ar@{..>}[d]^{\lambda} & M' \ar[r] \ar[d]^{\varphi} & 0 \\
      0 \ar[r] & \cL \ar[r]                      & N \ar[r]_{\pi}                                    & M \ar[r]                   & 0,
      }
    \end{equation}
    where the left vertical arrow is the canonical morphism \(\varphi^*\cL \to \cL\) of
    systems over \(\varphi\).

    \begin{Const}[Pullback]
      The pullback is simply played by the fibre product, meaning
      \[
        \varphi^*N \;:=\; N \times_M M' \;=\; \bigl\{( n,m') \;\bigm|\; m' \in M',\; n
        \in N,\; \varphi(m') \;=\; \pi(n)\bigr\},
      \]
      with multiplication \((n_1, m'_1) (n_2,m'_2) := ( n_1 n_2.\, m_1'm_2')\). We have
      the projection \(\sigma\colon \varphi^*N \to M'\), given by \(\sigma(n, m') :=
      m'\) and the actions
      \[ x \pt (n, m') \;:=\; (x \pt n,\, m') \]
      for each fibre \((n, m') \in \sigma^{-1}(m')\). Here \(x \in (\varphi^*\cL)(m') =
      \cL(\varphi(m'))\). Given a cleavage \(\kappa\) of \(\cE\), we have the cleavage
      \(\kappa'(m') := (\kappa(\varphi(m')), \, m')\) of the pullback.
    \end{Const}

    \begin{Pro} \label{pro:pullback_welldef}
      In the above construction, \(\varphi^*N\) is a monoid,
      \[
        \varphi^*\cE \;:=\; \bigl(0 \to \varphi^*\cL \to \varphi^*N
        \xto{\sigma} M' \to 0\bigr)
      \]
      is a precartesian coextension of \(M'\) by \(\varphi^*\cL\), and \(\kappa'\) is a
      cleavage of \(\varphi^*\cE\).
    \end{Pro}

    \begin{proof}
      It is clear that \(\varphi^*N\) is a commutative monoid and that \(\sigma\) is a
      surjective morphism. To see that the action is well-defined, let \(x \in
      \cL(\varphi(m'))\) and \((n, m') \in \sigma^{-1}(m')\). We have \(x \pt n \in
      \pi^{-1}(\varphi(m'))\) by assumption and subsequently, \((x\pt n,\, m') \in
      \varphi^*N\). Of course, \(\sigma(x\pt n,m') = m'\).

      It also satisfies the compatibility condition, as for \(x \in
      \cL(\varphi(m'_1))\), \(y \in \cL(\varphi(m'_2))\), \((n_1, m'_1) \in
      \sigma^{-1}(m'_1)\), \((n_2, m'_2) \in \sigma^{-1}(m'_2)\), we have
      \begin{align*}
        (x \pt (n_1, m'_1))\cdot(y \pt (n_2, m'_2)) & = ( x\pt n_1,m'_1)\cdot(y\pt n_2,m'_2)                     \\
                                                    & = ( (x\pt n_1)(y\pt n_2),\, m'_1m'_2)                      \\
                                                    & = ( (m'_{2*}(x)\cdot m_{1_*}'(y))\pt (n_1n_2),\, m'_1m'_2) \\
                                                    & = (m'_{2*}(x)\cdot m_{1*}'(y))\pt ( n_1n_2,\,m'_1m'_2 )    \\
                                                    & = (m'_{2*}(x)\cdot m_{1*}'(y))\pt \bigl((n_1, m'_1)\cdot(n_2, m'_2)\bigr).
      \end{align*}

      Lastly, it remains to verify the precartesianness of our pullback construction.
      For this, we aim to show that \(\kappa'(m') = ( \kappa(\varphi(m')), m')\) is a
      basepoint of \(\sigma^{-1}(m')\), meaning, for any given \((n,m') \in
      \sigma^{-1}(m')\), we need to find a unique \(x\in \cL(\varphi(m'))\), such that
      \(x\pt \kappa'(m') = (n,m')\).

      As \(\pi(n) = \varphi(m')\) and \(\cE\) is precartesian, there exists a unique
      \(x \in \cL(\varphi(m'))\) with \(x \pt \kappa(\varphi(m')) = n\). But then \(x
      \pt \kappa'(m') = ( x\pt\kappa(\varphi(m'), m')) = (n, m')\), as desired.
      Uniqueness of \(x\) in the pullback follows form the uniqueness of \(x\) in
      \(\cE\).
    \end{proof}

    It is clear that \(\lambda\) is a monoid homomorphism and that \(\pi\circ\lambda =
    \varphi\circ\sigma\). The action is likewise respected, that is to say, we have
    \(\lambda(x\pt(n, m')) = x\pt\lambda(n,m')\), since both sides equal \(x\pt n\).
    Hence, Diagram~\eqref{eq:pullback_diagramme} is a commutative diagram of
    coextensions. Moreover, by construction \(\lambda(\kappa'(m')) = \lambda(
    \kappa(\varphi(m'),m')) = \kappa(\varphi(m'))\) holds. This is exactly the chosen
    basepoint of \(\cE\) over \(\varphi(m')\) and thus, \(\lambda\) sends basepoints to
    basepoints, so we have indeed constructed a precartesian morphism.

    Indeed, the pullback construction \(\cE \mapsto \varphi^*\cE\) is functorial for
    both, varying \(\cE\) and \(\varphi\):

    \begin{Pro} \label{prop:pullback_functor}
      \phantom{.}
      \begin{enumerate}
        \item Let \((\id_\cL, h, \id_M)\colon \cE \to \cE'\) be a precartesian
              morphism. The map
              \[
                \varphi^*h\colon \varphi^*N \to \varphi^*N', \qquad
                (n,m') \;\mapsto\; (h(n),\,m'),
              \]
              is a precartesian morphism \((\id_{\varphi^*\cL}, \varphi^*h,
              \id_{M'})\colon \varphi^*\cE \to \varphi^*\cE'\).
        \item Let \(\psi\colon M'' \to M'\) be a further homomorphism. There is a
              natural isomorphism
              \[
                (\varphi\circ\psi)^*\cE \cong \psi^*(\varphi^*\cE),
                \qquad (n,m'') \;\mapsto\; ((n,\, \psi(m'')),\,m''))
              \]
              of precartesian morphism, which is itself precartesian.
      \end{enumerate}
    \end{Pro}

    \begin{proof}
      (1) Since the map between the monoids is the identity, meaning \(\id_M\), we have
      \(\pi'(h(n)) = \pi(n) = \varphi(m')\), so \(( h(n),m') \in \varphi^*N'\). It
      respects multiplication since
      \[
        \varphi^*h((n_1,m'_1)(n_2,m'_2))
        \;=\; ( h(n_1 n_2),m'_1 m'_2)
        \;=\; (
        h(n_1)h(n_2), m'_1 m'_2).
      \]
      That \(\sigma'\circ\varphi^*h = \sigma\) holds, and the action condition
      \[
        \varphi^*h(x\pt(n,m')) \;=\; (h(x\pt n),m') \;=\; (x\pt h(n),m') \;=\;
        x\pt\varphi^*h(n,m')
      \]
      are immediate. Basepoints are sent to basepoints as
      \[
        \varphi^*h(\kappa_1'(m')) \;=\; (h(\kappa(\varphi(m')),m') \;=\;
        (\kappa'(\varphi(m')),m') \;=\; \kappa_2'(m'),
      \]
      where \(\kappa_1', \kappa_2'\) are the induced cleavages of \(\varphi^*\cE\),
      \(\varphi^*\cE'\) respectively. Note that \(\kappa, \kappa'\) can always be
      chosen so that \(h(\kappa(m)) = \kappa'(m)\) for all \(m \in M\).

      (2) Both sides have the same underlying set, being \(\{(n,m'') \mid
      \varphi\psi(m'') = \pi(n)\}\). The given map is a bijection with inverse \((n,
      \psi(m''),m'') \mapsto(n,m'')\). It is easily verified that it is a monoid
      homomorphism, respects \(\sigma\) and the action. It also sends the basepoint \((
      \kappa(\varphi\psi(m'')), m'')\) of \((\varphi\psi)^*\cE\) to the basepoint \(((
      \kappa(\varphi(\psi(m''))), \psi(m'')), m'' ))\) of \(\psi^*(\varphi^*\cE)\).
    \end{proof}

\section{The category \(\Pcoex(M, \cL)\) and the Baer sum}\label{sec:pcoex}

  We now proceed to study the category \(\Pcoex(M, \cL)\), for a fixed monoid \(M\) and
  a fixed system of monoids \(\cL\). We have shown in \cite{p15} that when \(\cL\) is a
  system of abelian groups, \(\Pcoex(M, \cL)\) is a symmetric categorical group. We
  will show in Theorem~\ref{thm:monoidal_Pcoex_affine} that this remains mostly true
  for precartesian coextensions. Specifically, we show that \(\Pcoex(M, \cL)\) is a
  symmetric monoidal groupoid.

\subsection{The Baer sum} \label{sec:baer_affine}

    Fix \(\cL\in\cHCM(M)\), precartesian coextensions \(\cE,\cE'\in\Pcoex(M,\cL)\), and
    cleavages \(\kappa,\kappa'\). We will use the pushforward and pullback
    constructions we just developed in the preceding two sections, to construct a
    symmetric monoidal structure on \(\Pcoex(M,\cL)\). This constructed operation will
    be called the Baer sum, much as in the abelian case.

    For \(\cE, \cE'\), denote \(\cE\times\cE' = (0\to \cL\times\cL\to N\times N'\to
    M\times M\to 0)\) to be the componentwise product coextension of \(M\times M\) by
    \(\cL\times\cL\), with basepoint \(\bigl(m,(\kappa(m),\kappa'(m))\bigr)\) over
    \(m\in M\).

    \begin{De} \label{def:box_product_precart}
      Define the \emph{box product} of \(\cE, \cE'\) as
      \[ \cE\boxtimes\cE' \;:=\; \Delta^*(\cE\times\cE'), \]
      where \(\Delta\colon M\to M\times M\) is the diagonal.
    \end{De}

    \begin{Le}
      \(\cE\boxtimes\cE' \in \Pcoex(M,\cL\times\cL)\).
    \end{Le}

    \begin{proof}
      \(\cE\times\cE'\) is a precartesian coextension of \(M\times M\) by
      \(\cL\times\cL\) (basepoints and their uniqueness properties are checked in each
      coordinate separately), and pull-back of a precartesian coextension along
      \emph{any} monoid homomorphism -- here the diagonal \(\Delta\) -- is again
      precartesian by Proposition~\ref{pro:pullback_welldef}, which uses nothing beyond
      precartesianness of its input.
    \end{proof}

    \begin{De}
      The \emph{Baer sum} of \(\cE,\cE'\in\Pcoex(M,\cL)\) is
      \[ \cE+\cE' \;:=\; \nabla_*(\cE\boxtimes\cE'), \]
      where \(\nabla\colon\cL\times\cL \to \cL\) is the fold map \(\nabla(x,x')=xx'\).
    \end{De}

    It is a direct consequence of Prospoition~\ref{pro:pushforward_morphism} that the
    Baer sum lands again in \(\Pcoex(M,\cL)\).

    \begin{Le} \label{le:pushforward_functorial_L}
      Let \(\alpha\colon\cL\to\cL'\) be a morphism of systems. The assignment
      \(\cE\mapsto\alpha_*\cE\) extends to a functor \(\alpha_*\colon\Pcoex(M,\cL)\to
      \Pcoex(M,\cL')\).
    \end{Le}

    \begin{proof}
      Let us expand the functor on morphisms. For \(h\colon\cE\to\cE'' \in
      \Pcoex(M,\cL)\), define
      \[ \alpha_*(h)\colon\alpha_*\cE\to\alpha_*\cE'', \qquad [y,n]\mapsto[y,h(n)]. \]
      We start by proving that \(\alpha_*(h)\) is well-defined on the equivalence
      classes. Let \((y,n)\sim(y_1,n_1)\), we must show that
      \((y,h(n))\sim(y_1,h(n_1))\) in \(\alpha_*\cE''\). Note that
      \((y,n)\sim(y_1,n_1)\) means \(\pi(n)=\pi(n_1)\) and \(y\cdot \alpha_m(\xi(n)) =
      y_1\cdot\alpha_m(\xi(n_1))\). Since \(h\) is a morphism with \(\lambda=\id_\cL\),
      it respects basepoints and the action. Denote \(m = \pi(n)\). It follows that
      \[
        h(\xi(n)\pt \kappa(m)) \;=\; \xi(n)\pt h(\kappa(m)) \;=\; \xi(n)\pt\kappa''(m),
      \]
      which, by uniqueness of \(\xi''\), now gives \(\xi''(h(n)) = \xi(n)\), and
      likewise \(\xi''(h(n_1)) = \xi(n_1)\).

      As \(h\) commutes with the projections, we have \(\pi''(h(n)) = \pi(n) = m =
      \pi(n_1) = \pi''(h(n_1))\), so that the pairs \((y, h(n))\) and \((y_1, h(n_1))\)
      do indeed lie in \(\widetilde{K''}\) and satisfy the first of the two conditions
      defining \(\sim\). For the second condition, we substitute the two identities we
      have just obtained and get
      \[
        y\cdot\alpha_m(\xi''(h(n))) \;=\; y\cdot\alpha_m(\xi(n)) \;=\;
        y_1\cdot\alpha_m(\xi(n_1)) \;=\; y_1\cdot\alpha_m(\xi''(h(n_1))),
      \]
      which is exactly what we needed. Hence \(\alpha_*(h)\) is well-defined.

      \medskip

      It remains to see that \((\id_{\cL'}, \alpha_*(h), \id_M)\) is a morphism of
      precartesian coextensions in the sense of Definition~\ref{def:morphism_precoext}.
      Using \(\pi''\circ h = \pi\) once more, we obtain
      \begin{eqnarray*}
        \alpha_*(h)\bigl([y,n]\cdot[z,n_1]\bigr) & = & \bigl[\pi(n_1)_*(y)\cdot\pi(n)_*(z),\; h(nn_1)\bigr]              \\
                                                 & = & \bigl[\pi''(h(n_1))_*(y)\cdot\pi''(h(n))_*(z),\; h(n)h(n_1)\bigr] \\
                                                 & = & \alpha_*(h)[y,n]\cdot\alpha_*(h)[z,n_1],
      \end{eqnarray*}
      and the identity \([e_{1_M}, 1_N]\) is sent to \([e_{1_M}, h(1_N)] = [e_{1_M},
      1_{N''}]\), so that \(\alpha_*(h)\) is a monoid homomorphism. Compatibility with
      the projections to \(M\) reads \(\sigma''(\alpha_*(h)[y,n]) = \pi''(h(n)) =
      \pi(n) = \sigma[y,n]\). Since the accompanying morphism of systems is the
      identity of \(\cL'\), the condition on the actions amounts to
      \[
        \alpha_*(h)\bigl(u\pt[y,n]\bigr) \;=\; [u\cdot y,\, h(n)] \;=\;
        u\pt\alpha_*(h)[y,n]
      \]
      for \(u\in\cL'(m)\), which holds by the very definition of the two actions.
      Finally, writing \(\kappa_1'\) and \(\kappa_2'\) for the induced cleavages of
      \(\alpha_*\cE\) and \(\alpha_*\cE''\) provided by
      Lemma~\ref{le:pushforward_welldef}, basepoints are sent to basepoints, as
      \[
        \alpha_*(h)(\kappa_1'(m)) \;=\; [e_m,\, h(\kappa(m))] \;=\; [e_m,\,
        \kappa''(m)] \;=\; \kappa_2'(m),
      \]
      where the middle equality is the compatibility of the chosen cleavages,
      Equation~\eqref{eq:precartesian_kappa}.

      Functoriality is now immediate. The identity of \(\cE\) is sent to the map
      \([y,n]\mapsto[y,n]\), which is the identity of \(\alpha_*\cE\), and for two
      composable morphisms \(h\) and \(g\) in \(\Pcoex(M,\cL)\) we have
      \[
        \alpha_*(g\circ h)[y,n] \;=\; [y,\, g(h(n))] \;=\; \alpha_*(g)\bigl([y,
        h(n)]\bigr) \;=\; \bigl(\alpha_*(g)\circ\alpha_*(h)\bigr)[y,n],
      \]
      which finishes the proof.
    \end{proof}

    \begin{Le} \label{le:pushforward_functorial_M}
      Let \(\varphi\colon M'\to M\) be a monoid homomorphism. We have the induced
      functor \(\varphi^*\colon\Pcoex(M,\cL)\to\Pcoex(M',\varphi^*\cL)\).
    \end{Le}

    \begin{proof}
      This is an amalgamation of previous results: On objects, we set
      \(\cE\mapsto\varphi^*\cE\), which lands in \(\Pcoex(M',\varphi^*\cL)\) by
      Proposition~\ref{pro:pullback_welldef}. On morphisms, we map
      \(h\colon\cE\to\cE''\in \Pcoex(M,\cL)\) to \(\varphi^*(h) := \varphi^*h \colon
      \varphi^*\cE\to\varphi^*\cE''\in \Pcoex(M',\varphi^*\cL)\) by
      Proposition~\ref{prop:pullback_functor}(1).

      All that is left to check is that this assignment respects identities and
      composition. Since \(\varphi^*h\) acts as \((n,m')\mapsto(h(n), m')\), leaving
      the second coordinate untouched, the identity of \(\cE\) is sent to the identity
      of \(\varphi^*\cE\), and for two composable morphisms \(h\) and \(g\) we have
      \[
        \varphi^*(g\circ h)(n,m') \;=\; ( g(h(n)), \,m') \;=\;
        \varphi^*(g)\bigl(h(n), m'\bigr) \;=\;
        \bigl(\varphi^*(g)\circ\varphi^*(h)\bigr)(n,m'),
      \]
      as desired.
    \end{proof}

    \begin{Th} \label{thm:monoidal_Pcoex_affine}
      The Baer sum \(+\) makes \(\Pcoex(M,\cL)\) into a symmetric monoidal groupoid,
      with unit \(\cL\rtimes M\).
    \end{Th}

    \begin{proof}
      That \(\Pcoex(M,\cL)\) is a groupoid is Corollary~\ref{cor:Pcoex_groupoid}. That
      \(+\) is functorial in each variable are Lemmas~\ref{le:pushforward_functorial_L}
      and~\ref{le:pushforward_functorial_M}. What remains to check is that it satisfies
      all required identities, such as associativity condition and symmetry.

      Symmetry is simply induced by the swap \(\tau(x,x')=(x',x)\), giving us the
      isomorphism \(\cE+\cE'\xto{\sim}\cE'+\cE\). It is associativie since both
      \((\cE+\cE')+\cE''\) and \(\cE+(\cE'+\cE'')\) are naturally isomorphic to the
      push-forward of the triple box product \(\Delta^*(\cE\times\cE'\times\cE'')\)
      along the diagonal \((x,y,z)\mapsto xyz\).

      The unit is given by the map
      \[
        \upsilon\colon \nabla_*\Delta^*(\cE\times(\cL\rtimes M)) \to \cE, \qquad
        [z,(n,(x,m))] \mapsto (z\cdot x)\pt n,
      \]
      with \(m = \pi(n)\). As we are in a groupoid, any morphism is an isomorphism and
      thus, so is \(\upsilon\).
    \end{proof}

\section{Cohomological classification}~\label{sec:cohomology_affine}

\subsection{The cocycle monoid}

    Fix \(M\) and \(\cL \in \cHCM(M)\). We define the algebraic data classifying
    precartesian coextensions. Following Patchkoria's monoid cohomology
    \cite{patchkoria}, we group the alternating terms in the definitions of the cochain
    by sign and put them on aether side of the equality, to avoid the need for
    inverses. We also use the multiplicative notation as we are working with monoids
    and not abelian groups.

    We will use the notation along the lines of \(f: M \to \cL\) and treat it almost
    like a function. Of course, this is not directly a map, as \(\cL\) is a not a set.
    What we mean by that is, of course, the "nearest sensible thing", meaning, for each
    \(m\in M\), \(f(m)\in \cL(m)\). see also \cite{p15} for the same convention. We
    will also assume that such ''maps'' are reduced, meaning that the function take
    value \(1\) if one of the variables is \(1_M\).

    \begin{De}
      Define the commutative monoids under (pointwise) multiplication
      \begin{align*}
        \sC^0(M, \cL) & := \bigl\{ g \colon M \to \cL \;\bigm|\; g(m) \in \cL(m) \;\forall m \bigr\},                                    \\
        \sC^1(M, \cL) & := \bigl\{ f \colon M \times M \to \cL \;\bigm|\; f(a,b) \in \cL(ab),\;\; f(a,b) = f(b,a) \;\forall a,b \bigr\}, \\
        \sC^2(M, \cL) & := \bigl\{ h \colon M \times M \times M \to \cL \;\bigm|\; h(a,b,c) \in \cL(abc),\;\; h(a,b,c) \cdot h(b,c,a) = h(b,a,c) \;\forall a,b,c \bigr\}.
      \end{align*}
    \end{De}

    We would now like to define maps between these monoids and take the cohomology.
    This, however, is not as straightforward for monoids as for abelian groups.

    \begin{De}
      Define the \(0\)-th dimenional cohomology monoid \(\sD^0(M,\cL)\) to be the
      collection of all reduced ''maps'' \(g:M\to \cL\) for which
      \[ f(ab) \;=\; a_*f(b)b_*f(a) \]
      hold for all \(a,b\in M\).
    \end{De}
    \begin{De}
      Define the \emph{1-cocycle monoid} as
      \[
        \sZ^1(M, \cL) \;:=\;
        \left\{
        \begin{array}{lcl}
          f \in \sC^1(M, \cL) & \bigm| & a_*(f(b,c)) \cdot f(a,bc) \;=\; f(ab,c) \cdot c_*(f(a,b))\quad \forall a,b,c \in M \\[2pt]
                              &        & f(1_M, a) \;=\; e_a \quad \forall a \in M
        \end{array}
        \right\}.
      \]
    \end{De}

    \begin{Rem}[Relation to the Grillet complex]
      Note that for monoids, unlike for groups, the unit needs a little more attention,
      and is not automatically respected. Hence, the added condition \(f(1_M, a) \;=\;
      e_a\). However, when the cochain monoids happen to be (abelian) groups, the above
      still simplifies exactly to the Grillet complex.
    \end{Rem}

    \begin{Pro} \label{prop:cocycle_coboundary_submonoids}
      \(\sZ^1(M, \cL)\) is a submonoids of \(\sC^1(M, \cL)\).
    \end{Pro}

    \begin{proof}
      We need to show that \(\sZ^1(M, \cL)\) is closed under pointwise multiplication.
      Let \(f_1, f_2 \in \sZ^1\). Using that \(a_*\) is a monoid homomorphism and that
      \(\cL(abc)\) is commutative, we can write the following computation:
      \begin{align*}
        a_*((f_1 \cdot f_2)(b,c)) \cdot (f_1 \cdot f_2)(a,bc)
         & = a_*(f_1(b,c)) \cdot a_*(f_2(b,c)) \cdot f_1(a,bc) \cdot f_2(a,bc)                         \\
         & = \bigl[a_*(f_1(b,c)) \cdot f_1(a,bc)\bigr] \cdot \bigl[a_*(f_2(b,c)) \cdot f_2(a,bc)\bigr] \\
         & = \bigl[f_1(ab,c) \cdot c_*(f_1(a,b))\bigr] \cdot \bigl[f_2(ab,c) \cdot c_*(f_2(a,b))\bigr] \\
         & = f_1(ab,c) \cdot f_2(ab,c) \cdot c_*(f_1(a,b)) \cdot c_*(f_2(a,b))                         \\
         & = (f_1 \cdot f_2)(ab,c) \cdot c_*((f_1 \cdot f_2)(a,b)).
      \end{align*}
      The identity element of \(\sC^1\) is \(e\colon M \times M \to \cL\) with \(e(a,b)
      = e_{ab}\), which clearly satisfies the cocycle condition and \(e(1_M, a) =
      e_a\), so \(e \in \sZ^1\). That it is the unit, meaning \((f_1 \cdot f_2)(1_M, a)
      = f_1(1_M,a) \cdot f_2(1_M,a) = e_a \cdot e_a = e_a\) always holds, is likewise
      immediate.\\[2pt]
    \end{proof}

\subsection{The classifying category \texorpdfstring{\(\fC(M,\cL)\)}{C(M,L)}}
    \label{sec:classifying_cat}

    This section will introduce the category \(\fC(M,\cL)\) which will allow us to
    classifies precartesian coextensions.

    \begin{De} \label{def:classifying_cat}
      Define \(\fC(M, \cL)\) to be the category whose objects are elements of
      \(\sZ^1(M, \cL)\), and for \(f, f'\in \sZ^1(M, \cL)\), morphisms are elements \(g
      \in \sC^0(M, \cL)\), such that \(g(1_M) = e_{1_M}\) and
      \begin{equation} \label{eq:morphism_condition}
        f'(a,b) \cdot g(ab) \;=\; b_*(g(a)) \cdot a_*(g(b)) \cdot f(a,b)
        \qquad \forall\, a, b \in M.
      \end{equation}
      The composition is given by pointwise multiplication (\((g' \circ g)(a) := g(a)
      \cdot g'(a)\)), and the identity morphism on \(f\) is \(e \in \sC^0(M, \cL)\),
      which is defined by \(e(a) := e_a\) for all \(a \in M\).
    \end{De}

    \begin{Le} \label{le:C_is_category}
      The composition in Definition~\ref{def:classifying_cat} is well define, making
      \(\fC(M, \cL)\) indeed a category.
    \end{Le}

    \begin{proof}
      Let \(g \colon f'' \to f'\) and \(g' \colon f' \to f\) be morphisms and consider
      \((g' \circ g) := g \cdot g'\colon f'' \to f\). The identity condition holds as
      \((g \cdot g')(1_M) = g(1_M) \cdot g'(1_M) = e_{1_M} \cdot e_{1_M} = e_{1_M}\).
      For Condition~\eqref{eq:morphism_condition}, we write
      \begin{align*}
        f''(a,b) \cdot (g \cdot g')(ab) & = f''(a,b) \cdot g(ab) \cdot g'(ab)                                        \\
                                        & = b_*(g(a)) \cdot a_*(g(b)) \cdot f'(a,b) \cdot g'(ab)                     \\
                                        & = b_*(g(a)) \cdot a_*(g(b)) \cdot b_*(g'(a)) \cdot a_*(g'(b)) \cdot f(a,b) \\
                                        & = b_*(g(a) \cdot g'(a)) \cdot a_*(g(b) \cdot g'(b)) \cdot f(a,b)           \\
                                        & = b_*((g \cdot g')(a)) \cdot a_*((g \cdot g')(b)) \cdot f(a,b). \qedhere
      \end{align*}
    \end{proof}

    The category \(\fC(M, \cL)\) equivalent to the groupoid \([\sC^0(M, \cL)
    \xto{\partial^0} \sZ^1(M, \cL)]\) as given in \cite{p15} in the particular case
    when \(\cL\) is taking values in abelian groups. It is, however, not a groupoid in
    general. Thus, it makes sense to examine under what conditions morphism inside
    \(\fC(M, \cL)\) are isomorphisms.

    \begin{Le} \label{le:g_invertible}
      A morphism \(g: f'\to f\) is an isomorphism in \(\fC(M, \cL)\) if and only if
      \(g(a)\in \cL(a)^\times\) for all \(a\in M\).
    \end{Le}

    \begin{proof}
      This is straightforward as composition in \(\fC(M,\cL)\) is the pointwise product
      in \(\cL(a)\), and the identity morphism is \(a \mapsto e_a\). A morphism \(g
      \colon f \to f'\) is subsequently invertible in \(\fC(M,\cL)\) if and only if
      \(g(a) \in \cL(a)^\times\) for every \(a \in M\), with the inverse being
      \(g^{-1}(a) := g(a)^{-1}\).
    \end{proof}

\subsection{Classification theorem}

    \begin{Le} \label{le:cocycle_extension}
      Let \(f\in \sZ^1(M, \cL)\) be a cocyle. There exists a coextension \({}_fM \in
      \Pcoex(M, \cL)\) given by
      \[ {}_f M \;:=\; \{(x, a) \mid a \in M,\; x \in \cL(a)\}, \]
      with multiplication
      \[ (x, a)(y, b) \;:=\; (b_*(x) \cdot a_*(y) \cdot f(a,b),\; ab). \]
    \end{Le}

    \begin{proof}
      We first show that \({}_fM\) is a commutative monoid. Commutativity is clear and
      follows form the commutativity of the monoids \(\cL(ab)\). The identity is
      \((e_{1_M}, 1_M)\). For associativity, we have
      \begin{eqnarray*}
        [{(x,a)(y,b)}](z,c) & = & (b_*(x) \cdot a_*(y) \cdot f(a,b),\; ab)(z,c)                                \\
                            & = & (c_*(b_*(x) \cdot a_*(y) \cdot f(a,b)) \cdot (ab)_*(z) \cdot f(ab,c),\; abc) \\
                            & = & (c_*(b_*(x)) \cdot c_*(a_*(y)) \cdot c_*(f(a,b)) \cdot (ab)_*(z) \cdot f(ab,c),\; abc),
      \end{eqnarray*}
      and
      \begin{eqnarray*}
        (x,a)[{(y,b)(z,c)}] & = & (x,a)(c_*(y) \cdot b_*(z) \cdot f(b,c),\; bc)                                \\
                            & = & ((bc)_*(x) \cdot a_*(c_*(y) \cdot b_*(z) \cdot f(b,c)) \cdot f(a,bc),\; abc) \\
                            & = & ((bc)_*(x) \cdot (ac)_*(y) \cdot (ab)_*(z) \cdot a_*(f(b,c)) \cdot f(a,bc),\; abc).
      \end{eqnarray*}
      Since \(\cL(abc)\) is commutative and \((bc)_* = c_* \circ b_*\), \((ac)_* = c_*
      \circ a_*\), the two expressions agree if and only if
      \[ c_*(f(a,b)) \cdot f(ab,c) \;=\; a_*(f(b,c)) \cdot f(a,bc). \]
      This, however, is precisely the cocycle condition for \(f \in \sZ^1(M, \cL)\).
      Hence, \({}_fM\) is a commutative monoid.

      We define the map
      \[ \pi_f \colon {}_f M \to M \,\qtext{given by}\, \pi_f(x, a) \;=\; a. \]
      This is clearly a surjective monoid homomorphism.

      Next, we need to define an \(\cL\)-action on the fibres of \(\pi_f\). We do so by
      the rule \(z \pt (x, m) = (z \cdot x, m)\). As
      \[
        z \pt (z' \pt (x, m)) \;=\; z \pt (z' \cdot x, m) \;=\; (z \cdot z' \cdot x, m)
        \;=\; (z \cdot z') \pt (x, m),
      \]
      holds for all \(z, z' \in \cL(m)\), and \(e_m \pt (x, m) = (e_m \cdot x, m) = (x,
      m)\), we see that this is indeed an action.

      Lastly, we need to see that the precartesian condition is satisfied. Let \(m \in
      M\). Then \((e_m, m)\) is a basepoint of \(\pi^{-1}_f(m)\). Indeed, for any \((y,
      m)\), \(y\) is the unique element for which \(y \pt (e_m, m) = (y, m)\) holds.
      Thus, we have constructed a precartesian coextension
      \[
        {}_f\cE \;=\; (0 \to \cL \to {}_fM \xto{\pi_f} M \to 0) \;\in\; \Pcoex(M, \cL).
        \qedhere
      \]
    \end{proof}

    Given \(f, f'\in \sZ^1(M,\cL)\) we have just constructed their associated monoids
    \(M_f\) and \(M_{f'}\). The next question to explore is, what a map \(g: f'\to f\)
    corresponds to in terms of its associated monoids.

    \begin{Le} \label{le:associated_mon_hom}
      A map \(g: f'\to f\) induces a monoids homomorphism
      \[
        \gamma(g)\colon {}_{f'}M \to {}_fM, \qqtext{given by} (x,a)\mapsto (x\cdot
        g(a),a),
      \]
      compatible with the projections to \(M\) and with the \(\cL\)-action on the
      fibres.
    \end{Le}

    \begin{proof}
      We have
      \begin{eqnarray*}
        \gamma(g)(x,a) \cdot \gamma(g)(y,b) & = & (x \cdot g(a),\, a) \cdot (y \cdot g(b),\, b)                  \\
                                            & = & (b_*(x \cdot g(a)) \cdot a_*(y \cdot g(b)) \cdot f(a,b),\; ab) \\
                                            & = & (b_*(x) \cdot b_*(g(a)) \cdot a_*(y) \cdot a_*(g(b)) \cdot f(a,b),\; ab),
      \end{eqnarray*}
      and on the other hand,
      \begin{eqnarray*}
        \gamma(g)((x,a)(y,b)) & = & \gamma(g)(b_*(x) \cdot a_*(y) \cdot f'(a,b),\; ab) \\
                              & = & (b_*(x) \cdot a_*(y) \cdot f'(a,b) \cdot g(ab),\; ab).
      \end{eqnarray*}
      Thus, for them to agree, we need
      \[ b_*(g(a)) \cdot a_*(g(b)) \cdot f(a,b) \;=\; f'(a,b) \cdot g(ab) \]
      to hold. But this is exactly Condition~\eqref{eq:morphism_condition}. That it
      respects the unit is clear. Commutativity with \(\pi_f, \pi_{f'}\) and with the
      action are immediate from the formula for \(\gamma(g)\), since \(\pi_f(\gamma(g)
      (x,a)) = a = \pi_{f'}(x,a)\) and \(\gamma(g)(z\pt(x,a)) = (z\cdot x\cdot g(a),a)
      = z\pt\gamma(g)(x,a)\) for \(z\in\cL(a)\).
    \end{proof}

    \begin{Le} \label{le:iso_determined_by_val}
      The tripple \((\id_\cL,\gamma(g),\id_M)\) is precartesian if and only if \(g\) is
      an isomorphism in \(\fC(M, cL)\).
    \end{Le}

    \begin{proof}
      Recall that \(g\) is an isomorphism if and only if \(g(a) \in \cL(a)^\times\) for
      every \(a \in M\). Hence, all we have to do is check that this is precisely the
      condition under which \(\gamma(g)\) respects basepoints.

      By definition, \(\gamma(g)(x,a) = (x \cdot g(a), a)\). Thus \(\gamma(g)(e_a,a) =
      (g(a), a)\), with \((e_a, a)\) a basepoint of \({}_fM\). The description of
      basepoints of \({}_{f'}M\) in Lemma~\ref{le:cartesian_criterion} implies that
      this point is a basepoint in \({}_{f'}M\) if and only if \(g(a) \in
      \cL(a)^\times\).
    \end{proof}

    For a category \(\fC\), denote by \(\fC^\times\) its core, that is to say, the full
    subcategory of \(\fC\) whose morphisms are its isomorphisms. In particular, denote
    by \(\fC(M,\cL)^\times\) the core of \(\fC(M,\cL)\). Another way we can describe
    this is to say that its objects are \(\sZ^1(M,\cL)\) and morphisms are \(g \in
    \sC^0(M, \cL)\) with \(g(a) \in \cL(a)^\times\) for all \(a\). We have the
    following theorem, which is one of the main theorems of this paper:

    \begin{Th} \label{thm:classification_pcoex}
      There is an equivalence of categories
      \[ \fC(M,\cL)^{\times} \;\simeq\; \Pcoex(M, \cL), \]
      given on objects by \(f \mapsto {}_fM\) and on morphisms by \(g \mapsto
      \gamma(g)\). In particular, the following hold:
      \begin{enumerate}
        \item Let \(\sim\) be an equivalence relation on \(\sZ^1(M,\cL)\),
              where \(f \sim f'\) exactly when there is an \emph{isomorphism} \(f \to
              f'\) in \(\fC(M,\cL)\). We have
              \[ \pi_0(\Pcoex(M,\cL)) \simeq \sZ^1(M,\cL)/\sim, \]
        \item \(\End_{\fC(M,\cL)}(e) \simeq \sD^0(M,\cL)\), where \(e\in\sZ^1(M,\cL)\)
              is the unit cocycle and in particular,
              \[
                \pi_1(\Pcoex(M,\cL) \simeq \sD^0(M,\cL)^\times.
              \]
        \item When all \(\cL(m)\) are abelian groups, this recovers the equivalence of
              symmetric categorical groups from~\cite[Theorem~4.4]{p15}.
      \end{enumerate}
    \end{Th}

    \begin{proof}
      By Lemma~\ref{le:associated_mon_hom}, a morphism \(g \colon f' \to f\) of
      \(\fC(M,\cL)\) induces a well-defined monoid homomorphism \(\gamma(g)\colon
      {}_{f'}M \to {}_fM\), which is compatible with the projections and the action.
      Lemma~\ref{le:iso_determined_by_val} says that \(\gamma(g)\) is precartesian
      exactly when \(g\) is an isomorphism of \(\fC(M,\cL)\). Hence \(\gamma\)
      restricts to a functor \(\fC(M,\cL)^\times \to \Pcoex(M,\cL)\). We show that this
      restriction is full and faithful and essentially surjective.

      Full and faithful: Let \(\zeta\colon {}_{f'}M \to {}_fM\) be any morphism of
      \(\Pcoex(M,\cL)\). The commutativity of the appropriate diagram forces
      \(\pi_f(\zeta(x,a)) = a\), and thus \(\zeta(x,a) = (h(x,a), a)\) for some
      \(h(x,a) \in \cL(a)\). The compatibility of the action of \(\cL\) on the fibres
      means that \(\zeta(z \pt (x,a)) = z \pt \zeta(x,a)\). This gives us \(h(z \cdot
      x, a) = z \cdot h(x,a)\) and so \(h(x,a) = x \cdot h(e_a, a)\). Setting \(g(a) :=
      h(e_a, a) \in \cL(a)\), gives us \(\gamma(g) = \zeta\), with
      Condition~\eqref{eq:morphism_condition} coming from the fact that \(\zeta\) is in
      particular a monoid homomorphism. As \(\zeta\) sends basepoints to basepoints (it
      is a morphism of \(\Pcoex(M,\cL)\)), Lemma~\ref{le:iso_determined_by_val} forces
      \(g(a) \in \cL(a)^\times\) for every \(a\), so \(g\) is the (unique) morphism of
      \(\fC(M,\cL)^\times\) with \(\gamma(g) = \zeta\). This proves that our functor is
      full and faithful.

      Essential surjectivity: We take \(0 \to \cL \to N \xto{\pi} M \to 0 \in
      \Pcoex(M,\cL)\) and let \(\kappa\) be a cleavage. For all \(a, b \in M\),
      \(\kappa(ab)\) and \(\kappa(a)\kappa(b)\) lie in the same fibre \(\pi^{-1}(ab)\)
      and thus, by the definition of the cleavage, there is a unique \(f(a,b) \in
      \cL(ab)\) with
      \[ \kappa(a)\kappa(b) \;=\; f(a,b) \pt \kappa(ab). \]
      We claim that this \(f\) is our desired cocycle, meaning \(f\in \sZ^1(M, \cL)\)
      with \(N \cong {}_fM\) in \(\Pcoex(M,\cL)\). This is easily verified and we leave
      it to the reader. We just mention that the commutativity of \(N\) gives the
      symmetry of \(f\), the associativity of \(N\) gives the cocycle condition, and
      the map \(\varphi \colon N \to {}_fM\) given by \(\varphi(n) = (\xi(n),
      \pi(n))\), where \(\xi(n) \in \cL(\pi(n))\) is the unique element with \(\xi(n)
      \pt \kappa(\pi(n)) = n\), is a bijective homomorphism. Moreover
      \(\varphi(\kappa(m)) = (e_m,m)\), which is precisely the basepoint of \({}_fM\)
      over \(m\). Hence, \(\varphi\) sends basepoints to basepoints, and being
      bijective, so does its inverse. It follows that \(\varphi\) is an isomorphism of
      \(\Pcoex(M,\cL)\), proving \(N \cong {}_fM\) as precartesian coextensions,
      proving our desired equivalence of categories.

      \medskip

      Claims~(1) and~(3) are direct consequences of the equivalence constructed.

      To see Claim~(2), let \(\cL \rtimes M = {_e\cE}\), where \(e \in \sZ^1(M,\cL)\)
      is the unit section defined by \(e(a,b) = e_{ab}\) for all \(a, b \in M\). Being
      an equivalence of categories, an endomorphism of \(\cL \rtimes M\) in
      \(\Pcoex(M,\cL)\) naturally corresponds to a unique endomorphism of \(e\) in
      \(\fC(M,\cL)\), meaning \(g \colon e \to e\). Substituting \(f = f' = e\) into
      Condition~\eqref{eq:morphism_condition} gives
      \[ e(a,b) \cdot g(ab) \;=\; b_*(g(a)) \cdot a_*(g(b)) \cdot e(a,b) \]
      and thus
      \[ g(ab) \;=\; b_*(g(a)) \cdot a_*(g(b)), \]
      as \(e(a,b) = e_{ab}\), the unit. This is exactly the condition defining
      \(\sD^0(M,\cL)\). Lemma~\ref{le:g_invertible} gives our desired result for
      automorphisms.
    \end{proof}

    As there is an equivalence of categories, all structures are transpherable. In
    particular, the Baer sum should translate in a natural way to the classifying
    category. We claim that it is nothing but the monoidal
    (Proposition~\ref{prop:cocycle_coboundary_submonoids}) product.

    \begin{Pro} \label{prop:baer_cocycle_affine}
      Let \(f,f'\in\sZ^1(M,\cL)\). There is a natural isomorphism of precartesian
      coextensions
      \[ {}_fM + {}_{f'}M \;\cong\; {}_{f\cdot f'}M, \]
      where \(f\cdot f'\) denotes the pointwise product in the commutative monoid
      \(\sZ^1(M,\cL)\).
    \end{Pro}

    \begin{proof}
      Since \(\Pcoex(M, \cL)\) is a groupoid, all we have to do is show the existence
      of any morphism
      \[ \Theta\colon {}_fM + {}_{f'}M \longrightarrow {}_{f\cdot f'}M. \]
      It will by default be an isomorphism. The morphism is fairly automatic once we
      decupher what the elements are, excplicitely. We start with the box-product, as
      the Baer sum is based on it. By Definition~\ref{def:box_product_precart}, the
      underlying monoid of \({}_fM\boxtimes{}_{f'}M = \Delta^*({}_fM\times{}_{f'}M)\)
      consists of pairs of elements of laying over the same element of \(M\). This can
      be expressed as triples \((x,x',m)\) with \(m\in M\) and \(x,x'\in\cL(m)\). The
      product of \({}_fM\boxtimes{}_{f'}M\) is given by
      \[
        (x,x',m)\cdot(y,y',m') \;=\; \bigl(b_*(x)\cdot a_*(y)\cdot f(m,m'),
        b_*(x')\cdot a_*(y')\cdot f'(m,m'), mm'\bigr).
      \]
      The action of \((z,z')\in\cL(a)\times\cL(a)\) is given by
      \[ (z,z')\pt(x,x',m) \;=\; (z\cdot x,\, z'\cdot x', m), \]
      and the basepoint over \(m\) is \((e_m,e_m,m)\). As such, the unique element
      \(\xi((x,x',m))\) of \((\cL\times\cL)(m)\) carrying the basepoint to \((x,x',m)\)
      is \((x,x')\) itself. Passing to \(\nabla_*\), the invariant governing the
      equivalence relation of Construction~\ref{const:pushforward} becomes
      \[
        u\cdot\nabla_m\bigl(\xi(x,x',m)\bigr) \;=\; u\cdot x\cdot x' \;\in\; \cL(m).
      \]
      By definition, two classes \([u,(x,x',m)]\) and \([v,(y,y',m')]\) of
      \({}_fM+{}_{f'}M\) agree if and only if \(m = m'\), and these invariants
      coincide. We are now in a position to define our desired morphism as
      \[
        \Theta\bigl[u,(x,x',m)\bigr] \;:=\; \bigl(u\cdot x\cdot x',m\bigr),
      \]
      which is clearly well-defined. It also clearly respects the action, as
      \(w\pt[u,(x,x',m)] = [w\cdot u,(x,x',m)]\) is sent to \((w\cdot u\cdot x\cdot
      x',m) = w\pt\Theta[u,(x,x',m)]\) for every \(w\in\cL(m)\), and it sends the
      basepoint \([e_m,(e_m,e_m,m)]\) of \({}_fM+{}_{f'}M\) provided by
      Lemma~\ref{le:pushforward_welldef} to \((e_m,m)\), which is the basepoint of
      \({}_{f\cdot f'}M\) exhibited in Lemma~\ref{le:cocycle_extension}.

      All that remains is to see that \(\Theta\) is a monoid homomorphism. Unwinding
      the product of the pushforward and then that of the box product, we obtain
      \[
        [u,(x,x',m)]\cdot[v,(y,y',m')] \;=\; \bigl[m'_*(u)\cdot m_*(v),\;
        (x,x',m)\cdot(y,y',m')\bigr].
      \]
      Under \(\Theta\), its image is
      \[
        m'_*(u)\cdot m_*(v)\cdot m'_*(x)\cdot m_*(y)\cdot f(m,m')\cdot m'_*(x')\cdot
        m_*(y')\cdot f'(m,m') \;=\; m'_*(u\cdot x\cdot x')\cdot m_*(v\cdot y\cdot
        y')\cdot (f\cdot f')(m,m').
      \]
      But this is exactley the first coordinate of
      \[
        \bigl(u\cdot x\cdot x',m\bigr)\cdot\bigl(v\cdot y\cdot y',m'\bigr) \;=\;
        \Theta[u,(x,x',m)]\cdot\Theta[v,(y,y'm')]
      \]
      computed in \({}_{f\cdot f'}M\), so the two agree. The unit is respected as well,
      since \([e_{1_M},(e_{1_M},e_{1_M},1_M)]\) is sent to \((e_{1_M},1_M)\). Hence
      \(\Theta\) is a morphism, and therefore an isomorphism, of precartesian
      coextensions.

      To see that this is natural in \(f\) and \(f'\), let \(g\colon f_1\to f\) and
      \(g'\colon f_1'\to f'\) be isomorphisms in \(\fC(M,\cL)\). The Baer sum of the
      induced morphisms sends \([u,(x,x';a)]\) to \([u,(x\cdot g(m),\, x'\cdot
      g'(m),m)]\), and applying \(\Theta\) to this gives
      \[
        \bigl(u\cdot x\cdot g(m)\cdot x'\cdot g'(m),\; a\bigr) \;=\; \bigl(u\cdot
        x\cdot x'\cdot (g\cdot g')(m),m\bigr) \;=\; \gamma(g\cdot
        g')\bigl(\Theta[u,(x,x'm)]\bigr).
      \]
      As \(g\cdot g'\) is the composite of \(g\) and \(g'\) under the pointwise product
      of \(\sC^0(M,\cL)\), which is exactly the morphism \(f_1\cdot f_1'\to f\cdot f'\)
      corresponding to the pair \((g,g')\), the isomorphism \(\Theta\) is natural in
      both arguments.
    \end{proof}

    \begin{Cor} \label{cor:transport_monoidal_affine}
      We have an equivalence of symmetric monoidal groupoids
      \[ \Gamma\colon\fC(M,\cL)^\times\to\Pcoex(M,\cL) \]
      mapping the pointwise product to the Baer sum.
    \end{Cor}

\section{Versal precartesian coextensions}

\subsection{Definition and construction}

    \begin{De}
      A precartesian coextension \(0 \to \cB \to K \xto{q} M \to 0\) is called
      \emph{versal} if for every \(0 \to \cL \to N \xto{\pi} M \to 0 \in \Pcoex(M)\),
      there exists a morphism
      \[
        (\alpha, \psi, \id_M) \colon (0 \to \cB \to K \to M \to 0)
        \longrightarrow (0 \to \cL \to N \to M \to 0).
      \]
    \end{De}

    This can be constructed in the same manner as for the abelian case \cite{p15}.

    \begin{Const}[Versal precartesian coextension]
      Choose a free commutative monoid \(F\) with a surjective monoid homomorphism
      \(\pi_F \colon F \twoheadrightarrow M\). Consider the system \(\cB = \cB^{\pi_F}
      \in \cHCM(M)\) generated by symbols \(\delta_{f,g}\), where \(\pi_F(f) = \pi_F(g)
      = m \in M\), modulo the relations
      \[
        \delta_{f,f} \;=\; e_m, \qquad
        \delta_{ff_1, gg_1}
        \;=\; \pi_F(f_1)_*(\delta_{f,g}) \cdot \pi_F(f)_*(\delta_{f_1,g_1}).
      \]
      Of course, \(\pi_F(f_1) = \pi_F(g_1)\) as well. For \((a,m): m \to am \in
      \cHCM(M)\), we have \(a_* \colon \cB(m) \to \cB(am)\) given by \(\delta_{f,g}
      \mapsto \delta_{ff_0, gf_0}\), where \(f_0 \in \pi_F^{-1}(a)\) is a lift of
      \(a\).\\

      The coextension monoid \(K\) is subsequently defined to be the set of equivalence
      classes of pairs \((\omega, f)\), where \(f \in F\) and \(\omega \in
      \cB(\pi_F(f))\). We put an equivalence relation on this where \((\omega, f) \sim
      (\tau, g)\) if and only if
      \[
        \pi_F(f) \;=\; \pi_F(g) \qquad \,\text{and}\, \qquad \omega \cdot \tau^{-1}
        \;=\; \delta_{g,f}.
      \]
      The class of \((\omega, f)\) is denoted by \([\omega, f]\). The monoid structure
      of \(K\) come from declaring
      \[
        [\omega, f] \cdot [\omega_1, f_1] \;:=\; [\pi_F(f_1)_*(\omega) \cdot
        \pi_F(f)_*(\omega_1),\; ff_1].
      \]

      Finally, the projection \(q \colon K \to M\) is given by \(q[\omega, f] =
      \pi_F(f)\) with the action of \(u \in \cB(m)\) on the fibre \(q^{-1}(m)\) defined
      as \(u \pt [\omega, f] = [u \cdot \omega, f]\).
    \end{Const}

    \begin{Le}
      The above construction yields a versal precartesian coextension.
    \end{Le}

    The proof is exactly as in~\cite[p15]{p15} and we omit it.

\subsection{Exact sequence associated to a coextension}

    Throughout this section, let \(\cL \in \cHCM(M)\) and let \(\cE_0:= (0 \to \cB \to
    K \xto{q} M \to 0)\) denote a precartesian coextension.

    \begin{De}
      Denote by \(D(q, \cL) \subseteq \sD^0(M, \cL)\) the submonoid (see
      Lemma~\ref{le:dql_submonoid}) of all derivations \(\partial \colon K \to q^*\cL\)
      satisfying
      \begin{equation} \label{eq:D_condition}
        \partial((x \cdot y) \pt n) \cdot \partial(n)
        \;=\; \partial(x \pt n) \cdot \partial(y \pt n),
      \end{equation}
      for all \(n \in K\) and \(x, y \in \cB(q(n))\).
    \end{De}

    \begin{Le} \label{le:dql_submonoid}
      \(D(q, \cL)\) is a commutative monoid under pointwise multiplication.
    \end{Le}

    \begin{proof}
      As
      \[
        e((x \cdot y) \pt n) \cdot e(n) \;=\; e_{q(n)} \cdot e_{q(n)}
        \;=\; e(x \pt n) \cdot e(y \pt n)
      \]
      we have \(e \colon n \mapsto e_{q(n)} \in D(q, \cL)\) (unit). For \(\partial_1,
      \partial_2 \in D(q, \cL)\), we have
      \begin{align*}
        (\partial_1 \cdot \partial_2)((x \cdot y) \pt n) \cdot (\partial_1 \cdot \partial_2)(n) & = \partial_1((x \cdot y) \pt n) \cdot \partial_2((x \cdot y) \pt n) \cdot \partial_1(n) \cdot \partial_2(n) \\
                                                                                                & = \partial_1((x \cdot y) \pt n) \cdot \partial_1(n) \cdot \partial_2((x \cdot y) \pt n) \cdot \partial_2(n) \\
                                                                                                & = \partial_1(x \pt n) \cdot \partial_1(y \pt n) \cdot \partial_2(x \pt n) \cdot \partial_2(y \pt n)         \\
                                                                                                & = \partial_1(x \pt n) \cdot \partial_2(x \pt n) \cdot \partial_1(y \pt n) \cdot \partial_2(y \pt n)         \\
                                                                                                & = (\partial_1 \cdot \partial_2)(x \pt n) \cdot (\partial_1 \cdot \partial_2)(y \pt n),
      \end{align*}
      and hence, the product satisfies Condition~\eqref{eq:D_condition}, meaning \(D(q,
      \cL)\) is multiplicatively closed. Commutativity is clear (submonoid).
    \end{proof}

    \begin{Pro} \label{prop:exact_pcoex}
      We have the following maps:
      \begin{enumerate}
        \item An injective monoid homomorphism
              \[
                \eta \colon \sD^0(M, \cL) \to D(q, \cL), \qquad \eta(\delta) \;:=\;
                \delta \circ q.
              \]
        \item A map of sets
              \[
                \tau \colon \Hom_{\cHCM(M)}(\cB, \cL) \to \pi_0(\Pcoex(M, \cL)),
                \qquad \tau(\alpha) \;:=\; [\alpha_*\cE_0],
              \]
              where \(\alpha_*\cE_0\) denotes the pushforward coextension of \(\cE_0\)
              along \(\alpha\).
      \end{enumerate}
    \end{Pro}

    \begin{proof}
      (1) Let \(\delta \in \sD^0(M, \cL)\) and denote \(\partial := \delta \circ q\),
      which is a derivation of \(q^*\cL\) in the obvious way. Since by
      Lemma~\ref{le:pi_respects_action}, we have
      \[
        \partial((x \cdot y) \pt n) \cdot \partial(n)
        \;=\; \delta(q(n)) \cdot \delta(q(n))
        \;=\; \partial(x \pt n) \cdot \partial(y \pt n).
      \]
      It follows that Condition~\eqref{eq:D_condition} holds and thus, we indeed land
      in \(D(q, \cL)\). Moreover, for \(\delta, \delta' \in \sD^0(M, \cL)\) and \(n \in
      K\), we can write
      \begin{eqnarray*}
        \eta(\delta \cdot \delta')(n) & = & (\delta \cdot \delta')(q(n))           \\
                                      & = & \delta(q(n)) \cdot \delta'(q(n))       \\
                                      & = & \eta(\delta)(n) \cdot \eta(\delta')(n) \\
                                      & = & (\eta(\delta) \cdot \eta(\delta'))(n),
      \end{eqnarray*}
      and \(\eta(e)(n) = e_{q(n)} = e(n)\), making \(\eta\) a monoid homomorphism.

      For injectivity, let \(\eta(\delta) = e\), meaning \(\delta(q(n)) = e_{q(n)},\;
      \forall n \in K\). Since \(q \colon K \to M\) is surjective, every \(m \in M\) is
      of the form \(q(n)\) for some \(n \in K\). Hence \(\delta(m) = e_m\) for all \(m
      \in M\), so \(\delta = e\).

      (2) Given \(\alpha \in \Hom_{\cHCM(M)}(\cB, \cL)\), the pushforward
      \(\alpha_*\cE_0\) is a precartesian coextension of \(M\) by \(\cL\), by
      Section~\ref{sec:pushforward}. The only thing to show here is that its
      isomorphism class \([\alpha_*\cE_0] \in \pi_0(\Pcoex(M, \cL))\) depends only on
      \(\alpha\). But this is clear since if \(\alpha = \alpha'\), then \(\alpha_*\cE_0
      = \alpha'_*\cE_0\) as coextensions.
    \end{proof}

    \begin{Cor}
      The map \(\tau \colon \Hom_{\cHCM(M)}(\cB, \cL) \to \pi_0(\Pcoex(M, \cL))\) of
      Proposition~\ref{prop:exact_pcoex} is surjective.
    \end{Cor}

    \begin{proof}
      Let \(0 \to \cL \to N \xto{\pi} M \to 0\) be any precartesian coextension. By
      versality of \(\cE_0\), there exists a morphism \((\alpha, \psi, \id_M) \colon
      \cE_0 \to (0 \to \cL \to N \to M \to 0)\). By definition of the pushforward, this
      morphism factors through \(\alpha_*\cE_0\), \red{giving us an isomorphism
      \(\alpha_*\cE_0 \cong (0 \to \cL \to N \to M \to 0)\) in \(\Pcoex(M, \cL)\) --
      Why?}, proving surjectivity.
    \end{proof}

    \part{Precartesian coextensions of monoid functors}\label{part:functors}

\section{The category \(\Pcoex(\cF, \cL)\) and its cohomological classification}
  \label{sec:non-affine_pcoex}

\subsection{Definitions}

\subsubsection{Systems of commutative monoids over a monoid functor}

      Recall that a sheaf over a poset topology is equivalent to a contravariant
      functor over the underlying poset. If the reader is unfamiliar with this, we
      direct him/her to \cite{p15}, which talks a little more about this, and includes
      sources for ridged proofs.

      Let \(P\) be a finite poset and \(\cF \colon P^{op} \to \CMon\) a contravariant
      functor in commutative monoids. This is also called a presheaf of monoids. In
      this paper, and when working in the theory of monoid schemes, we will, however,
      generally refer to this as a sheaf. The reason is that, as already mentioned,
      presheaves of a post are eqivalent to sheaves over their associated poset
      topologies, and it is the latter that we care about in the theory of monoid
      schemes. (In the affine case, this would be the difference between \(\Spec(M)\)
      regarded as an orderd set (poset) \(\Spec(M)\) reganded as a topological sapce
      under the Zariski topology.) We denote the stalk at \(p \in P\) by \(\cF_p\) and
      the structural map for \(q \leq p\) by \(\phi^{\cF}_{p,q}\) or simply
      \(\phi_{p,q} \colon \cF_p \to \cF_q\).

      \begin{De}
        A \emph{system of commutative monoids over \(\cF\)} consists of the following
        datum:
        \begin{itemize}
          \item A system \(\cL_p \in \cHCM(\cF_p)\) for each \(p \in P\);
          \item A morphism \(\alpha_{p,q} \colon \cL_p \to \phi_{p,q}^*\cL_q\) in
                \(\cHCM(\cF_p)\) for each \(q \leq p\).
        \end{itemize}
        These must satisfy \(\alpha_{p,p} = \id\) and the cocycle condition
        \(\alpha_{p,r} = \phi_{p,q}^*(\alpha_{q,r}) \circ \alpha_{p,q}\) for \(r \leq q
        \leq p\).

        A morphism of such systems is a collection of morphisms \(\psi_p \colon \cL_p
        \to \cL'_p\) compatible with all \(\alpha_{p,q}\)'s. The resulting category is
        denoted by \(\cHCM(\cF)\).
      \end{De}

      Henceforth, throughout this section, \(P\) denotes a poset, \(\cF: P\to \CMon\) a
      sheaf of commutative monoids and \(\cL \in \cHCM(\cF)\) is a system of
      commutative monoids over \(\cF\).

\subsubsection{Precartesian coextensions of monoid functors}

      \begin{De} \label{def:pcoex_functor}
        A \emph{precartesian coextension of \(\cF\) by \(\cL\)}, written \(0 \to \cL
        \to \cG \to \cF \to 0\), is a sheaf of commutative monoids \(\cG\) over \(P\)
        with a morphism of sheaves \(\pi \colon \cG \to \cF\) such that for each \(q
        \leq p \in P\):
        \begin{itemize}
          \item[(i)] \(0 \to \cL_p \to \cG_p \xto{\pi_p} \cF_p \to 0\)
                is a precartesian coextension, and
          \item[(ii)] the square
                \[
                  \xymatrix{
                    0 \ar[r]                                 & \cL_p \ar[r] \ar@{=>}[d]_{\alpha_{p,q}} & \cG_p
                    \ar[r]^{\pi_p} \ar[d]_{\phi_{p,q}^{\cG}} & \cF_p \ar[r]
                    \ar[d]^{\phi_{p,q}^{\cF}}                & 0 \\
                    0 \ar[r]                                 & \cL_q \ar[r]                            & \cG_q \ar[r]_{\pi_q} & \cF_q \ar[r] & 0
                  }
                \]
                is a morphism of precartesian coextensions.
        \end{itemize}
        A \emph{morphism} of such coextensions is a natural transformation \(h \colon
        \cG \to \cG'\) for which \((\id_{\cL_p}, h_p, \id_{\cF_p})\) is a morphism in
        the sense of Definition~\ref{def:morphism_precoext}, for each \(p\). In
        particular, each \(h_p\) respects basepoints. The resulting category is denoted
        by \(\Pcoex(\cF, \cL)\).
      \end{De}

      \begin{Rem}
        We note the following regarding the notation in
        Definition~\ref{def:pcoex_functor}: Since \(\cL_p\) and \(\cL_q\) are functors
        defined over different categories (\(\mathbb{H}(\cF_p)\) and
        \(\mathbb{H}(\cF_q)\)), the transformation \(\alpha_{p,q}\) is in actuallity a
        natural transformation \(\cL_p\to \phi_{p,q}^*( \cL_q)\), to be fully precise.
        We avoid complicating the notation and simply write
        \(\alpha_{p,q}:\cL_p\to\cL_q\), since it can not cause confusion.
      \end{Rem}

\subsection{Monoidal structure}

    \begin{Cor}
      \(\Pcoex(\cF,\cL)\) is a groupoid.
    \end{Cor}

    \begin{proof}
      As an isomorphism of sheaves is exactly a natural transformation that is an
      isomorphism at every stalk, the result follows from our affine discussion.
    \end{proof}

    We now establish the analogous monoidal structure for \(\Pcoex(\cF,\cL)\), before
    turning to the cocycle description of \(\Pcoex(\cF,\cL)\). As everything here is
    stalkwise, we simply have to check that Theorem~\ref{thm:monoidal_Pcoex_affine},
    applied at each stalk, assembles into a well-defined object and natural
    isomorphisms of sheaves, compatibly with the restriction maps \(\phi_{p,q}\).

    \begin{Th} \label{thm:monoidal_Pcoex_functor}
      For every \(\cL\in\cHCM(\cF)\), \(\Pcoex(\cF,\cL)\) is a symmetric monoidal
      groupoid, with monoidal product the stalkwise Baer sum \((\cG+\cG')_p := \cG_p +
      \cG'_p\) and unit \(\cL\rtimes\cF\).
    \end{Th}

    \begin{proof}
      We know by Theorem~\ref{thm:monoidal_Pcoex_affine} that at each stalk \(p \in
      P\), the Baer sum makes \(\cG_p+\cG'_p \in \Pcoex(\cF_p,\cL_p)\) into a symmetric
      monoidal groupoid. But as the Baer sum is built out of pullbacks, pushforwards
      and the fold map, all of which respect filtered colimits and limits, and thus
      gluing, the result follows.
    \end{proof}

\subsection{Conditions on the global cocycle monoid}
    \label{sec:global_grillet_pcoex}

    We now proceed to define the cochains and cocycles, and subsequently the
    classification theorem. Much as in the affine case, we avoid inverses through
    grouping the positive and negative parts.

    Let \(0 \to \cL \to \cG \to \cF \to 0\) be a precartesian coextension of sheaves
    and fix \(q \leq p\) in \(P\). By Theorem~\ref{thm:classification_pcoex} and
    Condition~(i) of Definition~\ref{def:pcoex_functor}, the stalk \(\cG_p\) is
    isomorphic to
    \[ {_{f(p)}}\cF_p \;=\; \{(a,x) \mid x \in \cF_p,\ a \in \cL_p(x)\}, \]
    for some \(f(p,-,-) \in \sZ^1(\cF_p, \cL_p)\). Recall that the multiplication is
    given by
    \[ (a,x)(b,y) \;=\; (y_*(a) \cdot x_*(b) \cdot f(p,x,y),\, xy). \]
    We fix such an identification for each \(p \in P\) and write elements of \(\cG_p\)
    as pairs \((a,x)\) with \(x \in \cF_p\) and \(a \in \cL_p(x)\). The basepoint over
    \(x\) becomes \((e_x, x)\) and the \(\cL_p(x)\)-action is given by \(z \pt (a,x) =
    (z \cdot a, x)\) under this isomorphism.\\

    By Condition~(ii) of Definition~\ref{def:pcoex_functor} (and in the same notation),
    we know that
    \begin{equation} \label{eq:equivariance_action}
      \phi^{\cG}_{p,q}(z \pt n) \;=\; \alpha_{p,q}(z) \pt \phi^{\cG}_{p,q}(n)
    \end{equation}
    must hold for all \(n \in \cG_p\) and \(z \in \cL_p(\pi_p(n))\). Moreover,
    commutativity of the right square gives us
    \[ \pi_q \circ \phi^{\cG}_{p,q} \;=\; \phi_{p,q} \circ \pi_p. \]

    Let \((a,x) \in \cG_p\). Then \(\pi_p(a,x) = x\) and thus,
    \(\pi_q(\phi^{\cG}_{p,q}(a,x)) = \phi_{p,q}(x)\) must hold. This means that the
    second coordinate of \(\phi^{\cG}_{p,q}(a,x)\) is determined to be
    \(\phi_{p,q}(x)\). Let us now focus on the first coordinate and denote it by some
    function \(h_x \colon \cL_p(x) \to \cL_q(\phi_{p,q}(x))\). We have
    \begin{equation} \label{eq:restriction}
      \phi^{\cG}_{p,q}(a,x) \;=\; \bigl(h_x(a), \phi_{p,q}(x)\bigr).
    \end{equation}
    Plugging \((a,x)\) for \(n\) in Equation~\eqref{eq:equivariance_action} gives us
    \begin{eqnarray*}
      \phi^{\cG}_{p,q}(z \pt (a,x)) & = & \alpha_{p,q}(z) \pt \phi^{\cG}_{p,q}(a,x)             \\
                                    & = & \alpha_{p,q}(z) \pt \bigl(h_x(a), \phi_{p,q}(x)\bigr) \\
                                    & = & (\alpha_{p,q}(z) \cdot h_x(a), \phi_{p,q}(x)).
    \end{eqnarray*}
    On the other hand,
    \begin{eqnarray*}
      \phi^{\cG}_{p,q}(z \pt (a,x)) & = & \phi^{\cG}_{p,q}(z \cdot a, x)) \\
                                    & = & \bigl(h_x(z \cdot a), \phi_{p,q}(x)\bigr)
    \end{eqnarray*}
    and thus we may deduce
    \[ \alpha_{p,q}(z) \cdot h_x(a) \;=\; h_x(z \cdot a). \]
    Now, letting \(a = e_x\) and \(z = z \cdot e_x\), gives
    \[
      \alpha_{p,q}(z) \cdot h_x(e_x) \;=\; h_x(z)
      \qquad \forall\, z \in \cL_p(x).
    \]
    In other words, we have shown that \(h_x\) is completely determined by its value at
    the identity \(e_x\). Hence, Equation~\eqref{eq:restriction} becomes
    \[
      \phi^{\cG}_{p,q}(a,x) \;=\; \bigl(\alpha_{p,q}(a) \cdot h_x(e_x),
      \phi_{p,q}(x)\bigr).
    \]
    Defining
    \[ g(p,q,x) \;:=\; h_x(e_x) \;\in\; \cL_q\bigl(\phi_{p,q}(x)\bigr) \]
    gives us
    \begin{equation} \label{eq:phiG_form}
      \phi^{\cG}_{p,q}(a,x) \;=\; \bigl(\alpha_{p,q}(a) \cdot g(p,q,x),\;
      \phi_{p,q}(x)\bigr),
    \end{equation}
    where \((a,x) \in \cG_p\). Every \(\phi_{p,q}^{\cG}\) is thus of the above form for
    a unique function \(g(p,q,-) \colon \cF_p \to \cL_q\).\\

    Our aim now is to understand under which conditions on \(\bigl(f(p,-,-),\,
    g(p,q,-)\bigr)\) \(\phi^{\cG}_{p,q}(a,x)\) is a sheaf.

\subsubsection*{The conditions on \(g\)}

      Recall that for \(\phi^{\cG}_{p,q}\) to be a sheaf means that the restriction
      maps must be composable monoid homomorphisms. Moreover, identities must be
      respected, which, for monoids unlike for groups, is always an additional
      constraint. By this, we mean that the identity restriction must be the identity
      and that the restriction maps must map identities to identities. We will start to
      explore what each of these conditions mean. Once we have assembled all these, we
      will finally define the category which encodes this information in the following
      section.

      \begin{Le} \label{lem:phiG_monoid_hom}
        The map \(\phi^{\cG}_{p,q}\) is a monoid homomorphism if and only if
        \begin{equation} \label{eq:Ziv}
          \alpha_{p,q}(f(p,x,y)) \cdot g(p,q,xy) \;=\;
          \phi_{p,q}(y)_*(g(p,q,x)) \cdot \phi_{p,q}(x)_*(g(p,q,y)) \cdot f(q,\phi_{p,q}(x),\phi_{p,q}(y))
        \end{equation}
        holds for all \(x, y \in \cF_p\).
      \end{Le}

      \begin{proof}
        Let \(x, y \in \cF_p\) be fixed and take arbitrary \(a \in \cL_p(x)\), \(b \in
        \cL_p(y)\). Using \eqref{eq:phiG_form}, we have
        \begin{eqnarray}
          \phi^{\cG}_{p,q}\bigl((a,x)(b,y)\bigr) & = & \phi^{\cG}_{p,q}\left(\bigl(y_*(a) \cdot x_*(b) \cdot f(p,x,y),\; xy\bigr)\right)                                                                                                                 \notag \\
                                                 & = & \Bigl(\alpha_{p,q}(y_*(a)) \cdot \alpha_{p,q}(x_*(b)) \cdot \alpha_{p,q}(f(p,x,y)) \cdot g(p,q,xy),\; \phi_{p,q}(xy)\Bigr)                                                                        \notag \\
                                                 & = & \Bigl(\underbrace{\phi_{p,q}(y)_*(\alpha_{p,q}(a)) \cdot \phi_{p,q}(x)_*(\alpha_{p,q}(b))}_{=:\,U(a,b)} \;\cdot\; \underbrace{\alpha_{p,q}(f(p,x,y)) \cdot g(p,q,xy)}_{=:\,W_1}, \label{eq:phiG_hom_LHS} \\
                                                 &   & \ \ \phi_{p,q}(x)\phi_{p,q}(y)\Bigr).                                                                                                                                                             \notag
        \end{eqnarray}
        Here we used the fact that as \(\alpha_{p,q} \colon \cL_p \to
        \phi_{p,q}^*\cL_q\) is a morphism of systems, we have
        \[ \alpha_{p,q}(y_*(a)) \;=\; \phi_{p,q}(y)_*(\alpha_{p,q}(a)). \]

        \medskip\noindent On the other hand,
        \begin{eqnarray}
          \phi^{\cG}_{p,q}(a,x) \cdot \phi^{\cG}_{p,q}(b,y) & = & \bigl(\alpha_{p,q}(a) \cdot g(p,q,x),\, \phi_{p,q}(x)\bigr) \cdot \bigl(\alpha_{p,q}(b) \cdot g(p,q,y),\, \phi_{p,q}(y)\bigr) \notag \\
                                                            & = & \Bigl(\phi_{p,q}(y)_*\bigl(\alpha_{p,q}(a) \cdot g(p,q,x)\bigr)\cdot                                                          \notag \\
                                                            &   & \ \ \phi_{p,q}(x)_*\bigl(\alpha_{p,q}(b) \cdot g(p,q,y)\bigr) \cdot                                                           \notag \\
                                                            &   & \ \ f(q,\phi_{p,q}(x),\phi_{p,q}(y)),\; \phi_{p,q}(x)\phi_{p,q}(y)\Bigr)                                                      \notag \\
                                                            & = & \Bigl(\underbrace{\phi_{p,q}(y)_*(\alpha_{p,q}(a)) \cdot \phi_{p,q}(x)_*(\alpha_{p,q}(b))}_{=\,U(a,b)} \cdot \label{eq:phiG_hom_RHS} \\
                                                            &   & \ \ \underbrace{\phi_{p,q}(y)_*(g(p,q,x)) \cdot \phi_{p,q}(x)_*(g(p,q,y)) \cdot f(q,\phi_{p,q}(x),\phi_{p,q}(y))}_{=:\,W_2},  \notag \\
                                                            &   & \ \ \phi_{p,q}(x)\phi_{p,q}(y)\Bigr).                                                                                         \notag
        \end{eqnarray}

        \medskip\noindent Comparing \eqref{eq:phiG_hom_LHS} and\eqref{eq:phiG_hom_RHS}
        we claim that they equal exactly when \(W_1 = W_2\).

        Let \(U(a,b)\cdot W_1 = U(a,b) \cdot W_2\) and let \(a = e_x, b = y_y\). Thus,
        we have
        \begin{eqnarray*}
          U(e_x, e_y) & = & \phi_{p,q}(y)_*(\alpha_{p,q}(e_x)) \cdot \phi_{p,q}(x)_*(\alpha_{p,q}(e_y)) \\
                      & = & \phi_{p,q}(y)_*(e_{\phi_{p,q}(x)}) \cdot \phi_{p,q}(x)_*(e_{\phi_{p,q}(y)}) \\
                      & = & e_{\phi_{p,q}(xy)} \cdot e_{\phi_{p,q}(xy)}                                 \\
                      & = & e_{\phi_{p,q}(xy)},
        \end{eqnarray*}
        implying the claim.

        Decyphering what \(W_1\) and \(W_2\) were, we get
        \(\phi^{\cG}_{p,q}\bigl((a,x)(b,y)\bigr) = \phi^{\cG}_{p,q}(a,x) \cdot
        \phi^{\cG}_{p,q}(b,y)\) if and only if
        \[
          \alpha_{p,q}(f(p,x,y)) \cdot g(p,q,xy) \;=\; \phi_{p,q}(y)_*(g(p,q,x)) \cdot
          \phi_{p,q}(x)_*(g(p,q,y)) \cdot f(q,\phi_{p,q}(x),\phi_{p,q}(y)),
        \]
        which is exactly~\eqref{eq:Ziv}.
      \end{proof}

      Next, we consider what conditoin must be satisfied for the composition of
      \(\phi^G\) to hold.

      \begin{Le} \label{lem:phiG_transitive}
        We have \(\phi^{\cG}_{q,r} \circ \phi^{\cG}_{p,q} = \phi^{\cG}_{p,r}\) if and
        only if
        \begin{equation} \label{eq:Ziii}
          \alpha_{q,r}(g(p,q,x)) \cdot g(q,r,\phi_{p,q}(x)) \;=\; g(p,r,x)
        \end{equation}
        holds for all \(x \in \cF_p\), where \(r \leq q \leq p \in P\).
      \end{Le}

      \begin{proof}
        Let \(x \in \cF_p\) be fixed and \(a \in \cL_p(x)\) be arbitrary. Using
        Formula~\eqref{eq:phiG_form}, we have
        \begin{eqnarray}
          \phi^{\cG}_{q,r}\bigl(\phi^{\cG}_{p,q}(a,x)\bigr) & = & \phi^{\cG}_{q,r}\bigl(\alpha_{p,q}(a) \cdot g(p,q,x),\, \phi_{p,q}(x)\bigr)                                                                   \notag \\
                                                            & = & \Bigl(\alpha_{q,r}\bigl(\alpha_{p,q}(a) \cdot g(p,q,x)\bigr) \cdot g\bigl(q,r,\phi_{p,q}(x)\bigr),\; \phi_{q,r}\bigl(\phi_{p,q}(x)\bigr)\Bigr)\notag \\
                                                            & = & \Bigl(\alpha_{q,r}(\alpha_{p,q}(a)) \cdot \alpha_{q,r}(g(p,q,x)) \cdot g\bigl(q,r,\phi_{p,q}(x)\bigr),\; \phi_{p,r}(x)\Bigr)                  \notag \\
                                                            & = & \Bigl(\underbrace{\alpha_{p,r}(a)}_{=:\,U(a)} \;\cdot\; \underbrace{\alpha_{q,r}(g(p,q,x)) \cdot g(q,r,\phi_{p,q}(x))}_{=:\,W_1},\; \phi_{p,r}(x)\Bigr)
        \end{eqnarray}
        In the last step, we used the fact that \(\cL\) satisfies the cocycle condition
        by definition, meaning that \(\alpha_{q,r}(\alpha_{p,q}(a)) = \alpha_{p,r}(a)\)
        holds.

        \medskip\noindent On the other hand,
        \begin{equation}
          \phi^{\cG}_{p,r}(a,x) \;=\; \Bigl(\underbrace{\alpha_{p,r}(a)}_{=\,U(a)} \cdot
          \underbrace{g(p,r,x)}_{=:\,W_2},\;\; \phi_{p,r}(x)\Bigr).
        \end{equation}
        A similar argument to the above (taking \(a = e_x\)) gives us that
        \(\phi^{\cG}_{q,r}\bigl(\phi^{\cG}_{p,q}(a,x)\bigr) = \phi^{\cG}_{p,r}(a,x)\)
        if and only if \(W_1 = W_2\). That is to say, if and only if
        \[ \alpha_{q,r}(g(p,q,x)) \cdot g(q,r,\phi_{p,q}(x)) \;=\; g(p,r,x). \qedhere \]
      \end{proof}

      Next, we explore the identity and what it means to be respected.

      \begin{Le} \label{le:Z_norm_pp}
        We have \(\phi^{\cG}_{p,p} = \id\) if and only if \(g(p,p,x) = e_x\) for all
        \(x \in \cF_p\), \(p \in P\).
      \end{Le}

      \begin{proof}
        Setting \(p = q\) in Formulan~\eqref{eq:phiG_form}, we get
        \[
          \phi^{\cG}_{p,p}(a,x) \;=\; \bigl(\alpha_{p,p}(a) \cdot g(p,p,x),\,
          \phi_{p,p}(x)\bigr) \;=\; \bigl(a \cdot g(p,p,x),\, x\bigr).
        \]
        Thus \(\phi^{\cG}_{p,p} = \id\) is equivalent to \(a \cdot g(p,p,x) = a\) for
        every \(a \in \cL_p(x)\). In particular, this must hold for \(a = e_x\). We get
        \(e_x \cdot g(p,p,x) = e_x\), i.e.\ \(g(p,p,x) = e_x\), as desired.
      \end{proof}

      Just as in the affine case, where we had to impose the condition
      \(f(1_M,a)=e_a\), we have the following condition in the monoid sheaf case:

      \begin{Le} \label{le:Z_norm_1}
        We have \(\phi^{\cG}_{p,q}(1_{\cG_p}) = 1_{\cG_q}\) if and only if
        \(g(p,q,1_{\cF_p}) \;=\; e_{1_{\cF_q}}\).
      \end{Le}

      \begin{proof}
        By definition of \(\cG\), its unit is \(1_{\cG_p} = (e_{1_{\cF_p}},\,
        1_{\cF_p})\). As \(\alpha_{p,q}\) and \(\phi_{p,q}\) respect units (morphisms),
        Formula~\eqref{eq:phiG_form} with \(x=y=1_{\cF_p}\) gives us
        \[
          \phi^{\cG}_{p,q}(1_{\cG_p}) \;=\; \bigl(\alpha_{p,q}(e_{1_{\cF_p}}) \cdot
          g(p,q,1_{\cF_p}),\; \phi_{p,q}(1_{\cF_p})\bigr) \;=\;
          \bigl(g(p,q,1_{\cF_p}),\; 1_{\cF_q}\bigr).
        \]
        This is equivalent to \(g(p,q,1_{\cF_p}) = e_{1_{\cF_q}}\).
      \end{proof}

\subsection{The global cocycle monoid}

    We summarise our discussion above with the following definition:

    \begin{De} \label{def:Z1_pcoex_functor}
      The \emph{global cocycle monoid} \(\sZ^1(\cF, \cL)\) consists of pairs \((f, g)\)
      such that for every \(q \leq p \in P\):
      \begin{itemize}
        \item \(f(p,-,-) \in \sZ^1(\cF_p, \cL_p)\) and
        \item \(g(p,q,-)\) assigns to each \(x \in \cF_p\) an element \(g(p,q,x) \in
              \cL_q(\phi_{p,q}(x))\).
      \end{itemize}
      These must satisfy the following constraints:
      \begin{enumerate}
        \item \(\alpha_{q,r}(g(p,q,x)) \cdot g(q,r,\phi_{p,q}(x)) = g(p,r,x)\),
              \label{eq:global_cocycle_1}
        \item \(\alpha_{p,q}(f(p,x,y)) \cdot g(p,q,xy)
              \;=\; \phi_{p,q}(y)_*(g(p,q,x)) \cdot \phi_{p,q}(x)_*(g(p,q,y)) \cdot
              f(q,\phi_{p,q}(x),\phi_{p,q}(y))\), \label{eq:global_cocycle_2}
        \item \(g(p,p,x) = e_x\),
              \label{eq:global_cocycle_3}
        \item \(g(p,q,1_{\cF_p}) = e_{1_{\cF_q}}\).
              \label{eq:global_cocycle_4}
      \end{enumerate}
      The multiplication of \(\sZ^1(\cF, \cL)\) is induced pointwise. The identity
      element is the pair \(f(p,-,-) = e\) and \(g(p,q,-) = e_{\phi_{p,q}}\).
    \end{De}

    \begin{Le}
      \(\sZ^1(\cF, \cL)\) is a well-defined commutative monoid.
    \end{Le}

    \begin{proof}
      What we have to check is that the identity and multiplication respect all four
      conditions.

      It is straightforward to see that the identity pair, which is to say \(f(p,a,b) =
      e_{ab}\) and \(g(p,q,x) = e_{\phi_{p,q}(x)}\), respect the identities.

      For the product, take \((f_1,g_1), (f_2,g_2) \in \sZ^1(\cF,\cL)\) denote \((f,g)
      := (f_1 \cdot f_2,\, g_1 \cdot g_2)\). That \(f(p,-,-) = f_1(p,-,-) \cdot
      f_2(p,-,-) \in \sZ^1(\cF_p,\cL_p)\) for each \(p\) is
      Proposition~\ref{prop:cocycle_coboundary_submonoids}, applied pointwise. That
      Conditions~\eqref{eq:global_cocycle_4} and~\eqref{eq:global_cocycle_3} hold is
      clear. For Condition~\eqref{eq:global_cocycle_1} we have
      \begin{align*}
        \alpha_{q,r}\bigl(g(p,q,x)\bigr) \cdot g\bigl(q,r,\phi_{p,q}(x)\bigr) & = \alpha_{q,r}\bigl(g_1(p,q,x) \cdot g_2(p,q,x)\bigr) \cdot g_1\bigl(q,r,\phi_{p,q}(x)\bigr) \cdot g_2\bigl(q,r,\phi_{p,q}(x)\bigr)         \\
                                                                              & = \alpha_{q,r}(g_1(p,q,x)) \cdot \alpha_{q,r}(g_2(p,q,x)) \cdot g_1(q,r,\phi_{p,q}(x)) \cdot g_2(q,r,\phi_{p,q}(x))                         \\
                                                                              & = \bigl[\alpha_{q,r}(g_1(p,q,x)) \cdot g_1(q,r,\phi_{p,q}(x))\bigr] \cdot \bigl[\alpha_{q,r}(g_2(p,q,x)) \cdot g_2(q,r,\phi_{p,q}(x))\bigr] \\
                                                                              & = g_1(p,r,x) \cdot g_2(p,r,x)                                                                                                               \\
                                                                              & = g(p,r,x),
      \end{align*}
      where the fourth equality uses~\eqref{eq:Ziii} for \(g_1\) and for \(g_2\)
      individually.

      For Condition~\eqref{eq:global_cocycle_2}, we again compute directly to obtain
      \begin{align*}
        \alpha_{p,q}\bigl(f(p,x,y)\bigr) \cdot g(p,q,xy) & = \alpha_{p,q}\bigl(f_1(p,x,y) \cdot f_2(p,x,y)\bigr) \cdot g_1(p,q,xy) \cdot g_2(p,q,xy)                                    \\
                                                         & = \alpha_{p,q}(f_1(p,x,y)) \cdot \alpha_{p,q}(f_2(p,x,y)) \cdot g_1(p,q,xy) \cdot g_2(p,q,xy)                                \\
                                                         & = \bigl[\alpha_{p,q}(f_1(p,x,y)) \cdot g_1(p,q,xy)\bigr] \cdot \bigl[\alpha_{p,q}(f_2(p,x,y)) \cdot g_2(p,q,xy)\bigr]        \\
                                                         & = \Bigl[\phi_{p,q}(y)_*(g_1(p,q,x)) \cdot \phi_{p,q}(x)_*(g_1(p,q,y)) \cdot f_1(q,\phi_{p,q}(x),\phi_{p,q}(y))\Bigr]         \\
                                                         & \ \ \cdot \Bigl[\phi_{p,q}(y)_*(g_2(p,q,x)) \cdot \phi_{p,q}(x)_*(g_2(p,q,y)) \cdot f_2(q,\phi_{p,q}(x),\phi_{p,q}(y))\Bigr] \\
                                                         & = \phi_{p,q}(y)_*\bigl(g_1(p,q,x) \cdot g_2(p,q,x)\bigr) \cdot \phi_{p,q}(x)_*\bigl(g_1(p,q,y) \cdot g_2(p,q,y)\bigr)        \\
                                                         & \ \ \cdot f_1(q,\phi_{p,q}(x),\phi_{p,q}(y)) \cdot f_2(q,\phi_{p,q}(x),\phi_{p,q}(y))                                        \\
                                                         & = \phi_{p,q}(y)_*(g(p,q,x)) \cdot \phi_{p,q}(x)_*(g(p,q,y)) \cdot f(q,\phi_{p,q}(x),\phi_{p,q}(y)),
      \end{align*}
      where the fourth equality uses~\eqref{eq:Ziv} for \(g_1\) and \(g_2\)
      individually.
    \end{proof}

\subsection{Morphisms}

    We now wish to determine what morphisms of coextensions mean in this setting and
    what conditions they impose. Recall that in the affie case, a morphism between \(f,
    f'\in \sZ(M, \cL)\) corresponds to elements \(g\in \cC^0\) satisfying \(g(1_M) =
    e_{1_M}\) and
    \[
      f'(a,b) \cdot g(ab) \;=\; b_*(g(a)) \cdot a_*(g(b)) \cdot f(a,b)
      \qquad \forall\, a, b \in M.
    \]
    By Theorem~\ref{thm:classification_pcoex}, this is equivalent to a morphism between
    the coextensions
    \begin{equation} \label{eq:affine_pcoex_to_monoid_morphism}
      \gamma(g)\colon {}_{f'}M \to {}_fM \,\qtext{give by}\, (a,x)\mapsto (a\cdot
      g(x),x).
    \end{equation}
    The monoid sheaf version is essentially a gluing of this datum.

    In other words, for \((f,g), (f',g') \in \sZ^1(\cF,\cL)\), we will call a
    collection \(h_p \in \sC^0(\cF_p,\cL_p)\), \(p\in P\) a morphism if at every point,
    \(h_p: f(p,-,-) \longrightarrow f'(p,-,-)\) is an affine morphism. That is to say,
    \(h_p(1_{\cF_p}) = e_{1_{\cF_p}}\) and
    \[
      f'(p,a,b) \cdot h_p(ab) \;=\; b_*(h_p(a)) \cdot a_*(h_p(b)) \cdot f(p,a,b)
      \qquad \forall\, a,b \in \cF_p.
    \]
    However, we still need to make sure that the maps \(h_p\) are compatible.
    Compatibility is most naturally stated in terms of the associated monoids
    hommorphism, stated above in~\eqref{eq:affine_pcoex_to_monoid_morphism}. In our
    notation, \(h_p\) induce morphisms
    \[
      \gamma_{h_p} \colon \cG_p \longrightarrow \cG'_p, \qqtext{given by}
      \gamma_{h_p}(a,x) \;=\; \bigl(a \cdot h_p(x),\, x\bigr),
    \]
    \(\cG_p = {_{f(p)}}\cF_p\) and \(\cG'_p = {_{f'(p)}}\cF_p\) are the coextensions
    classified by \(f(p,-,-)\) and \(f'(p,-,-)\) respectively. Compatibility now means
    the commutativity of the diagrams
    \begin{equation} \label{dg:compatibility_pcoex}
      \xymatrix{
      \cG_p \ar[r]^{\gamma_{h_p}} \ar[d]_{\phi^{\cG}_{p,q}} & \cG'_p \ar[d]^{\phi^{\cG'}_{p,q}} \\
      \cG_q \ar[r]^{\gamma_{h_q}}                           & \cG'_q
      }
    \end{equation}
    for \(q \leq p \in P\). Of course, we want to express this in terms of the cochain
    elements, and the following lemma allows us to do just that:

    \begin{Le}
      In the notation just above the lemma, Diagram~\eqref{dg:compatibility_pcoex}
      commutes if and only if
      \begin{equation} \label{eq:mor_ii}
        \alpha_{p,q}(h_p(x)) \cdot g'(p,q,x) \;=\; g(p,q,x) \cdot h_q(\phi_{p,q}(x))
        \qquad \forall\, q \leq p,\; x \in \cF_p.
      \end{equation}
    \end{Le}

    \begin{proof}
      Fix \(q \leq p\) and \(x \in \cF_p\), let \(a \in \cL_p(x)\) be arbitrary, and
      take \((a,x) \in \cG_p\). Throughout our calculations, we will use
      \(\gamma_{h_p}(a,x) = (a \cdot h_p(x), x)\) and~\eqref{eq:phiG_form}.

      For the top-right, we have
      \begin{align*}
        \phi^{\cG'}_{p,q}\bigl(\gamma_{h_p}(a,x)\bigr) & = \phi^{\cG'}_{p,q}\bigl(a \cdot h_p(x),\, x\bigr)                                    \\
                                                       & = \Bigl(\alpha_{p,q}\bigl(a \cdot h_p(x)\bigr) \cdot g'(p,q,x),\; \phi_{p,q}(x)\Bigr) \\
                                                       & = \Bigl(\underbrace{\alpha_{p,q}(a)}_{=:\,U(a)} \cdot \underbrace{\alpha_{p,q}(h_p(x)) \cdot g'(p,q,x)}_{=:\,W_1},\; \phi_{p,q}(x)\Bigr),
      \end{align*}
      and for the bottom-left, we have
      \begin{align*}
        \gamma_{h_q}\bigl(\phi^{\cG}_{p,q}(a,x)\bigr) & = \gamma_{h_q}\bigl(\alpha_{p,q}(a) \cdot g(p,q,x),\, \phi_{p,q}(x)\bigr)              \\
                                                      & = \Bigl(\alpha_{p,q}(a) \cdot g(p,q,x) \cdot h_q(\phi_{p,q}(x)),\; \phi_{p,q}(x)\Bigr) \\
                                                      & = \Bigl(\underbrace{\alpha_{p,q}(a)}_{=\,U(a)} \cdot \underbrace{g(p,q,x) \cdot h_q(\phi_{p,q}(x))}_{=:\,W_2},\; \phi_{p,q}(x)\Bigr).
      \end{align*}
      Comparing these two, we are in the same boat as in
      Lemmas~\ref{lem:phiG_monoid_hom} and~\ref{lem:phiG_transitive}: We have
      \(U(a)\cdot W_1 = U(a) \cdot W_2\) and must deduce that \(W_1 = W_2\) under our
      assumptions. Again, following the same strategy as in the above lemmas, we
      evaluate at \(a = e_x\) and derive our desired result.
    \end{proof}

    All this finally allows us to define the classifying category:

\subsection{The classifying category \(\fC(\cF, \cL)\)}

    \begin{De}
      Define \(\fC(\cF, \cL)\) to be the category whose objects are elements of
      \(\sZ^1(\cF, \cL)\), and for two elements \((f,g) \to (f',g')\), morphisms are
      compatible families \(h = (h_p)_{p \in P}\) of morphisms of extensions, see
      Definition~\ref{def:classifying_cat}. By compatible we mean that
      \[
        \alpha_{p,q}(h_p(x)) \cdot g'(p,q,x) \;=\; g(p,q,x) \cdot h_q(\phi_{p,q}(x))
        \qquad \forall\, q \leq p,\; x \in \cF_p
      \]
      holds. Composition is simply the pointwise product: \((h \cdot h')_p(x) := h_p(x)
      \cdot h'_p(x)\) and the unit is \(e_p \colon x \mapsto e_x\) for each \(p\).
    \end{De}

    \begin{Le}
      \(\fC(\cF, \cL)\) is a well-defined category.
    \end{Le}

    \begin{proof}
      By Lemma~\ref{le:C_is_category}, all we really need to check is that the
      compatibility condition~\eqref{eq:mor_ii} is respected by the identities and the
      composition. That the identity map, meaning \((f,g) = (f',g')\) and \(h_p = e_p\)
      satisfies it is straightforward.

      For composition, let \(h \colon (f,g) \to (f',g')\) and \(h' \colon (f',g') \to
      (f'',g'')\) be morphisms in \(\fC(\cF,\cL)\). We must show \(h \cdot h' \colon
      (f,g) \to (f'',g'')\) satisfies~\eqref{eq:mor_ii}. We have
      \begin{align*}
        \alpha_{p,q}\bigl((h \cdot h')_p(x)\bigr) \cdot g''(p,q,x) & = \alpha_{p,q}\bigl(h_p(x) \cdot h'_p(x)\bigr) \cdot g''(p,q,x)                                                   \\
                                                                   & = \alpha_{p,q}(h_p(x)) \cdot \Bigl[\alpha_{p,q}(h'_p(x)) \cdot g''(p,q,x)\Bigr]                                   \\
                                                                   & = \alpha_{p,q}(h_p(x)) \cdot \Bigl[g'(p,q,x) \cdot h'_q(\phi_{p,q}(x))\Bigr]\tag{by \eqref{eq:mor_ii} for \(h'\)} \\
                                                                   & = \Bigl[\alpha_{p,q}(h_p(x)) \cdot g'(p,q,x)\Bigr] \cdot h'_q(\phi_{p,q}(x))                                      \\
                                                                   & = \Bigl[g(p,q,x) \cdot h_q(\phi_{p,q}(x))\Bigr] \cdot h'_q(\phi_{p,q}(x)) \tag{by \eqref{eq:mor_ii} for \(h\)}    \\
                                                                   & = g(p,q,x) \cdot \Bigl[h_q(\phi_{p,q}(x)) \cdot h'_q(\phi_{p,q}(x))\Bigr]                                         \\
                                                                   & = g(p,q,x) \cdot (h \cdot h')_q(\phi_{p,q}(x)),
      \end{align*}
      which is exactly~\eqref{eq:mor_ii} for \(h \cdot h'\).
    \end{proof}

\subsection{Classification theorem}\label{sec:classifying_thm_functor}

    We now wish to classify \(\Pcoex(\cF,\cL)\) through \(\fC(\cF,\cL)\). It is no
    longer a direct equivalence, but we do have an embedding. We thus want to
    understand what the image looks like explicitly. We start by analysing what it
    means for a morphism of \(\fC(\cF,\cL)\) to be invertible on the cochain level.

    \begin{Le} \label{lem:mor_ii_invertible}
      A morphism \(h = (h_p)_p \colon (f,g) \to (f',g')\) of \(\fC(\cF,\cL)\) is
      invertible if and only if \(h_p(x) \in \cL_p(x)^\times\) for every \(p \in P\)
      and \(x \in \cF_p\). Its inverse is then given by \(h^{-1} := (h_p^{-1})_p\),
      with \(h_p^{-1}(x) := h_p(x)^{-1} \colon (f',g') \to (f,g)\).
    \end{Le}

    \begin{proof}
      Recall that composition in \(\fC(\cF,\cL)\) is the pointwise product. A such,
      invertibility of \(h\) immediately implies the invertibility of each \(h_p\) and
      thus, \(h_p(x) \in \cL(x)^\times\) for every \(x\) (and every \(p\in P\)).

      To prove the reverse, we have to prove that the pointwise inverse induces a
      global inverse. In other words, the only thing to show is that
      Condition~\eqref{eq:mor_ii} holds for the inverses, meaning that
      \[
        \alpha_{p,q}\bigl(h_p(x)^{-1}\bigr) \cdot g(p,q,x) \;=\; g'(p,q,x) \cdot
        h_q(\phi_{p,q}(x))^{-1}
      \]
      holds. We \emph{do} know that
      \[
        \alpha_{p,q}(h_p(x)) \cdot g'(p,q,x) \;=\; g(p,q,x) \cdot h_q(\phi_{p,q}(x)),
      \]
      holds and thus, we multiply both sides of this equation by the unit
      \(\alpha_{p,q}(h_p(x))^{-1} \cdot h_q(\phi_{p,q}(x))^{-1}\).
    \end{proof}

    Recall that we used the notation \(\fC(\cF,\cL)^\times\) to denote the core of
    \(\fC(\cF,\cL)\). By Lemma~\ref{lem:mor_ii_invertible}, \(\fC(\cF,\cL)^\times\) is
    thus equivalent to the wide subcategory of \(\fC(\cF,\cL)\) with the same objects
    \(\sZ^1(\cF,\cL)\) and morphisms \(h = (h_p)_p\) with \(h_p(x) \in
    \cL_p(x)^\times\) for every \(p \in P\), \(x \in \cF_p\).

    \begin{Th} \label{thm:classification_pcoex_functor}
      Let \(\cF\) be a sheaf of commutative monoids over a finite poset \(P\) and \(\cL
      \in \cHCM(\cF)\). There is an equivalence of categories
      \[
        \fC(\cF, \cL)^\times \;\simeq\; \Pcoex(\cF, \cL).
      \]
      Consequently, \(\pi_0(\Pcoex(\cF,\cL)) \cong \sZ^1(\cF,\cL)/\sim\), where \((f,g)
      \sim (f',g')\) exactly when there is an isomorphism between them in
      \(\fC(\cF,\cL)\), and
      \[
        \Aut_{\Pcoex(\cF, \cL)}(\cL \rtimes \cF) \;\cong\; \sD^0(\cF,\cL)^\times.
      \]
    \end{Th}

    \begin{proof}
      We construct a functor \(\Gamma \colon \fC(\cF,\cL)^\times \to \Pcoex(\cF,\cL)\)
      and show it is full and faithful, and essentially surjective. The construction of
      the sheaf is essentially our preliminary work in this section summarised.
      Specifically, we construct \(\Gamma\) as follows:

      On objects, it takes an element \((f,g) \in \sZ^1(\cF,\cL)\) to the sheaf of
      monoids \(\cG\) which is assembled by the stalks \(\cG_p := {_{f(p)}}\cF_p\) and
      striction maps \(\phi^{\cG}_{p,q}: \cG_p \to \cG_q\) for \(q \leq p\). This
      assembles to a sheaf by Lemmas~\ref{lem:phiG_monoid_hom},
      \ref{lem:phiG_transitive} and~\ref{le:Z_norm_pp}. It is a precartesian
      coextension of \(\cF\) by \(\cL\) as Condition~(i) of
      Definition~\ref{def:pcoex_functor} holds pointwise by
      Theorem~\ref{thm:classification_pcoex} and Condition~(ii) holds essentially by
      our construction of Condition~\eqref{eq:phiG_form} (see the discuessed text just
      above it). We thus set
      \[ \Gamma(f,g) \;:=\; (0 \to \cL \to \cG \to \cF \to 0) \in \Pcoex(\cF, \cL). \]

      On morphisms, take \(h = (h_p)_p \colon (f,g) \to (f',g')\) in
      \(\fC(\cF,\cL)^\times\), so that \(h_p(x) \in \cL_p(x)^\times\) for every \(p,x\)
      by Lemma~\ref{lem:mor_ii_invertible}. We have essentially already defined
      \(\Gamma(h)\) in Diagram~\ref{dg:compatibility_pcoex}: to restate it, we map
      \(h\) to \(\Gamma(h) := (\gamma_{h_p})_p\), where \(\gamma_{h_p} \colon \cG_p \to
      \cG'_p\) are as in Diagram~\ref{dg:compatibility_pcoex}. By
      Theorem~\ref{thm:classification_pcoex}, each \(\gamma_{h_p}\) is (as \(h_p\) is
      invertible) a genuine morphism of \(\Pcoex(\cF_p,\cL_p)\), so \(\Gamma(h)\) is a
      genuine morphism of \(\Pcoex(\cF,\cL)\), since composition and natural
      transformations reduce to pointwise products.

      \medskip \noindent To see that it is an equivalence, we start by checking that it
      is full and faithful. Recall that a morphism \(H \colon \cG \to \cG'\) of
      \(\Pcoex(\cF,\cL)\) is just a natural transformation of sheaves for which
      \((\id_{\cL_p}, H_p, \id_{\cF_p})\) is a morphism of precartesian coextensions at
      each stalk \(p\). By Theorem~\ref{thm:classification_pcoex}, each \(H_p\) is of
      the form \(\gamma_{h_p}\) for a unique \(h_p \in \fC(\cF_p,\cL_p)^\times\). These
      assemble into a unique morphism \((f,g) \to (f',g')\) in \(\fC(\cF,\cL)^\times\)
      with \(\Gamma((h_p)_p) = H\).

      Lastly, it remains to prove that it is essentially surjective. Let \(0 \to \cL
      \to \cG \xto{\pi} \cF \to 0 \in \Pcoex(\cF,\cL)\). We must construct an element
      \(\Gamma(f,g) \in \sZ^1(\cF,\cL)\) such that \(\cG \xto{\sim} \Gamma(f,g)\) are
      isomorphic as precartesian coextensions, in \(\Pcoex(\cF,\cL)\).

      In other words, for each \(q \leq p\), we have to find a cocycle \(f(p,-,-)\in
      \sZ^1(\cF_p,\cL_p)\), an isomorphism \(\varphi_p \colon \cG_p \xto{\sim}
      {}_{f(p)}\cF_p\) of \(\Pcoex(\cF_p,\cL_p)\), and an element \(g(p,q,x) \in
      \cL_q(\phi_{p,q}(x))\) for every \(x \in \cF_p\). These must satisfy \((f,g) \in
      \sZ^1(\cF,\cL)\) and
      \[
        \varphi_q \circ \phi^{\cG}_{p,q} \;=\; \phi^{\Gamma(f,g)}_{p,q} \circ
        \varphi_p.
      \]
      The collection of such \(f(p,-,-)\) and \(g(p,q,-)\) will then assemble to a
      \(\Gamma(f,g)\in \sZ(\cF, \cL)\) by Definition~\ref{def:Z1_pcoex_functor}, and
      the \(\varphi_p\) to a morphism of precartesian coextensions.

      \medskip \noindent At each \(p \in P\), choose a cleavage \(\kappa_p \colon \cF_p
      \to \cG_p\), meaning \(\kappa_p(x)\) is a basepoint of \(\pi_p^{-1}(x)\) for
      every \(x\), with \(\kappa_p(1_{\cF_p}) = 1_{\cG_p}\). Since \(\cG_p\) is a
      precartesian coextension of \(\cF_p\) by \(\cL_p\), every element of \(\cG_p\) is
      of the form \(a \pt \kappa_p(x)\) for some \(x \in \cF_p\) and \(a \in \cL_p(x)\)
      unique. The essential surjectivity of the equivalence in
      Theorem~\ref{thm:classification_pcoex}, applied to \(\cG_p\) with this cleavage,
      yields a cocycle \(f(p,-,-) \in \sZ^1(\cF_p,\cL_p)\), and an isomorphism of
      \(\Pcoex(\cF_p,\cL_p)\)
      \[
        \varphi_p \colon \cG_p \xto{\sim} {_{f(p)}}\cF_p, \qquad
        \varphi_p(\kappa_p(x)) \;=\; (e_x,x).
      \]

      \medskip \noindent Let \(q \leq p\). For \(x \in \cF_p\), denote \(y :=
      \phi_{p,q}(x)\). Since \(\cG_q\) is precartesian and \(\kappa_q(y)\) is a
      basepoint of \(\pi_q^{-1}(y)\), we have a unique element \(g(p,q,x) \in
      \cL_q(y)\) with
      \begin{equation}
        \phi^{\cG}_{p,q}(\kappa_p(x)) \;=\; g(p,q,x) \pt \kappa_q(y).
      \end{equation}
      This is our definition of \(g\), and we now proceed to verify that it is a valid
      choice: Take an element \(a \pt \kappa_p(x) \in \cG_p\), with \(a \in \cL_p(x)\).
      As \(\phi^{\cG}_{p,q}\) respects \(\alpha_{p,q}\) by
      Definition~\ref{def:pcoex_functor}(ii), we have
      \begin{eqnarray*}
        \phi^{\cG}_{p,q}\bigl(a \pt \kappa_p(x)\bigr) & = & \alpha_{p,q}(a) \pt \phi^{\cG}_{p,q}(\kappa_p(x))        \\
                                                      & = & \alpha_{p,q}(a) \pt \bigl(g(p,q,x) \pt \kappa_q(y)\bigr) \\
                                                      & = & \bigl(\alpha_{p,q}(a) \cdot g(p,q,x)\bigr) \pt \kappa_q(y),
      \end{eqnarray*}
      Applying \(\varphi_q\), we get
      \begin{eqnarray*}
        \varphi_q\Bigl(\phi^{\cG}_{p,q}\bigl(a \pt \kappa_p(x)\bigr)\Bigr) & = & \bigl(\alpha_{p,q}(a) \cdot g(p,q,x),\; y\bigr) \\
                                                                           & = & \phi^{\Gamma(f,g)}_{p,q}\bigl(\varphi_p(a \pt \kappa_p(x))\bigr),
      \end{eqnarray*}
      as \(\cL_q\) respects the action and \(\varphi_q(\kappa_q(y)) = (e_y,y)\)).
      Moreover \(\varphi_p(a\pt\kappa_p(x)) = (a,x)\) holds and
      \(\phi^{\Gamma(f,g)}_{p,q}\) acts as in Formula~\eqref{eq:phiG_form}. As \(a \pt
      \kappa_p(x)\) was a generic element of \(\cG_p\), the required compatibility
      \begin{equation} \label{eq:varphis_functor}
        \varphi_q \circ \phi^{\cG}_{p,q} = \phi^{\Gamma(f,g)}_{p,q} \circ \varphi_p
      \end{equation}
      is proved.

      \smallskip

      It remains to show that \((f,g) \in \sZ^1(\cF,\cL)\), meaning that the four
      conditions of Definition~\ref{def:Z1_pcoex_functor} hold. This is equivalent to
      showing that \(\Gamma(f,g)\) is a sheaf, by using
      Lemmas~\ref{lem:phiG_monoid_hom}, \ref{lem:phiG_transitive}, \ref{le:Z_norm_pp}
      and~\ref{le:Z_norm_1}, and Equation~\eqref{eq:phiG_form}. But that
      \(\Gamma(f,g)\) is a sheaf is almost trivial, since we can just use the sheaf
      strucutre of \(\cG\) to endow \(\Gamma(f,g)\) with one. Indeed, the collection
      \(\varphi = (\varphi_p)_p\) yields the desired isomorpihsm \(\varphi\colon \cG
      \iso \Gamma(f,g)\) as each \(\varphi_p\) is an isomorphsm in
      \(\Pcoex(\cF_p,\cL_p)\), and Relation~\ref{eq:varphis_functor} is precisely the
      fact that these isomorphism glue. We can rearrange these to get
      \begin{equation}
        \phi^{\Gamma(f,g)}_{p,q} \;=\; \varphi_q \circ \phi^{\cG}_{p,q} \circ
        \varphi_p^{-1}, \qquad q \le p,
      \end{equation}
      making \(\phi^{\Gamma(f,g)}_{p,q}\) the transport of \(\phi^{\cG}_{p,q}\) along
      the isomorphisms \(\varphi_p,\varphi_q\).
    \end{proof}

    Just as in the affine case, we can identify the Baer sum, stalkwise, on the cochain
    side, transporting the monoidal structure of
    Theorem~\ref{thm:monoidal_Pcoex_functor} to \(\fC(\cF,\cL)^\times\) along the
    equivalence \(\Gamma\) of Theorem~\ref{thm:classification_pcoex_functor}, rather
    than re-establishing and separately reconciling it.

    \begin{Cor}
      Let \((f,g),(f',g')\in\sZ^1(\cF,\cL)\). There is an isomorphism of precartesian
      coextensions
      \[
        \Gamma(f,g) + \Gamma(f',g') \;\cong\; \Gamma\bigl((f,g)\cdot(f',g')\bigr),
      \]
      natural in \((f,g),(f',g')\), where the product on the right is the pointwise
      product on \(\sZ^1(\cF,\cL)\). Consequently \(\Gamma\) upgrades to an equivalence
      of symmetric monoidal groupoids
      \[
        \bigl(\fC(\cF,\cL)^\times,\,\cdot\bigr) \;\simeq\;
        \bigl(\Pcoex(\cF,\cL),\,+\bigr),
      \]
      with the strict symmetric monoidal structure on the left given by the pointwise
      product on \(\sZ^1(\cF,\cL)\).
    \end{Cor}

    \begin{proof}
      At each stalk \(p\), \(\Gamma(f,g)_p = {}_{f(p)}\cF_p\) by construction
      (Theorem~\ref{thm:classification_pcoex_functor}), and
      Proposition~\ref{prop:baer_cocycle_affine} gives a natural isomorphism
      \({}_{f(p)}\cF_p+{}_{f'(p)}\cF_p\cong{}_{(f\cdot f')(p)}\cF_p\), where \((f\cdot
      f')(p) = f(p)\cdot f'(p)\) is the pointwise product of cocycles at \(p\), by
      definition of the multiplication on \(\sZ^1(\cF,\cL)\)
      (Definition~\ref{def:Z1_pcoex_functor}). By
      Theorem~\ref{thm:monoidal_Pcoex_functor}, the Baer sum on \(\Pcoex(\cF,\cL)\) is
      formed stalkwise, so these isomorphisms \(\Psi_p\) at each \(p\) already
      constructed in the proof of Proposition~\ref{prop:baer_cocycle_affine} assemble
      into an isomorphism of sheaves, because that construction was canonical (no
      choices beyond the cleavages already fixed by the cocycle presentations
      themselves) and hence automatically compatible with the restriction maps
      \(\phi^\cG_{p,q}\), which are likewise built stalkwise from the same data. The
      argument for the second statement is then identical to that of
      Corollary~\ref{cor:transport_monoidal_affine}, applied stalkwise.
    \end{proof}

\end{document}